\documentclass[11pt, reqno]{amsart}
\usepackage{geometry}                
\usepackage{graphicx}
\usepackage{amssymb}  
\usepackage{amsmath}
\usepackage{mathrsfs}
\usepackage{xcolor}
\usepackage{tikz-cd}
\usepackage[all]{xy}
\usepackage{tikz}
\usepackage{mathtools}
  \usepackage{enumitem}
  \usepackage{array}
  \usepackage{amsthm,amscd}
\newcommand{\sslash}{\mathbin{/\mkern-5mu/}}
\usetikzlibrary{arrows.meta,positioning}

\usepackage[colorlinks=true, pdfstartview=FitV,linkcolor=blue,citecolor=blue,urlcolor=blue]{hyperref}

    \advance \topmargin by -\headheight
    \advance \topmargin by -\headsep

    \evensidemargin \oddsidemargin
\numberwithin{equation}{section}

\def\cB{{\mathcal B}}

\def\cD{{\mathcal D}}
\def\cE{{\mathcal E}}

\def\cO{{\mathcal O}}

\def\ge{{\mathfrak e}}

\newcommand{\cdo}[1]{\mc{D}^{ch}_{#1}}

\newcommand{\coker}{\on{coker}}

\newcommand{\ssslash}{\mathbin{/\mkern-5mu/\mkern-5mu/}}

\def\cE{{\mathcal E}}

\newfont{\german}{eufm10}

\begin{document}

\theoremstyle{plain}
\newtheorem{maintheorem}{Main Theorem}
\newtheorem{theorem}{Theorem}[section]
\newtheorem{thm}[theorem]{Theorem}
  \newtheorem{prop}[theorem]{Proposition}
    \newtheorem{proposition}[theorem]{Proposition}
  \newtheorem{cor}[theorem]{Corollary}
    \newtheorem{corollary}[theorem]{Corollary}
  \newtheorem{lemma}[theorem]{Lemma}

    \newtheorem{conj}[theorem]{Conjecture}
    \theoremstyle{definition}
           \newtheorem{defn}[theorem]{Definition}
            \newtheorem{definition}[theorem]{Definition}

\theoremstyle{remark}
     \newtheorem{remark}[theorem]{Remark}
          \newtheorem{example}[theorem]{Example}
                   \newtheorem{ex}[theorem]{Example}

 \newcommand{\cx}{{\bf C}}
\newcommand{\la}{\langle}
\newcommand{\ra}{\rangle}
\newcommand{\res}{{\rm Res}}
\newcommand{\expp}{{\rm exp}}
\newcommand{\lgp}{\widehat{\rm G}( { F} ((t)) )}
\newcommand{\svee}{\scriptsize \vee}
\newcommand{\deff}{\stackrel{\rm def}{=}}
\newcommand{\cal}{\mathcal}
\renewcommand{\ge}{\geq}
\newcommand{\Coker}{\on{Coker}}

\newcommand{\on}{\operatorname}
\newcommand{\mc}{\mathcal}
\newcommand{\mf}{\mathfrak}
 \newcommand{\+}{\oplus}
  \newcommand{\g}{\mf{g}}
    \newcommand{\affg}{\widehat{\mf{g}}}
  \newcommand{\C}{\mathbb{C}}
  \newcommand{\isomap}{{\;\stackrel{_\sim}{\to}\;}}
  \newcommand{\bX}{\tilde{X}} 
    \newcommand{\bXd}{V^*/\!/G}
    \newcommand{\vac}{\mathbf{1}}
    \newcommand{\Z}{\mathbb{Z}}
      \newcommand{\Loc}[2]{\tilde{#2}_{#1}}
         \newcommand{\bigLoc}[2]{\widetilde{#2}_{#1}}
             \newcommand{\Floop}{F^{\on{loop}}}
    \newcommand{\Zhu}{\on{Zhu}}
      \newcommand{\Res}{\on{Res}}

\title{Faithful Localization of Vertex Algebra Modules}

\author{Tomoyuki Arakawa}

\address[Tomoyuki Arakawa]
{Okinawa Institute of Science and Technology,\newline 1919-1 Tancha, Onna-son, Kunigami-gun, Okinawa 904-0495, Japan}
\email{tomoyuki.arakawa@oist.jp}

\author{Xuanzhong Dai}
\address[Xuanzhong Dai]{School of Mathematics, Nanjing University,
Nanjing 210093, China}
\email{xzdai@nju.edu.cn}

\author{Bailin Song}
\address[Bailin Song]
{School of Mathematical Sciences, University of Science and Technology of China, Hefei, Anhui 230026, P. R. CHINA}
\email{bailinso@ustc.edu.cn}

\begin{abstract}
We establish a Faithfulness Theorem for vertex algebra modules, showing that two complementary geometric localizations can jointly detect global algebraic structure. We apply this principle to vertex algebras obtained by BRST reduction from classical invariant theory and identify them, under natural geometric hypotheses, with global sections of sheaves of chiral differential operators. As consequences, we prove vanishing of BRST cohomology in nonzero degrees and simplicity, and show that every finitely strongly generated module is a finite direct sum of copies of the vacuum module. In particular, the corresponding module category is equivalent to the category of finite-dimensional vector spaces. 
For families associated with classical groups, we identify the associated varieties with Hamiltonian reductions, proving quasi-lisse-ness and the predicted identification with Higgs branches of the corresponding supersymmetric gauge theories. This provides geometric realizations of a broad class of vertex algebras arising from three- and four-dimensional quantum field theories.
\end{abstract}
\maketitle

\section{Introduction}
Localization is a fundamental tool for relating representation theory to geometry: it replaces a global algebraic object by a geometric one that can often be studied locally. A basic difficulty, however, is that localization may lose information. In this paper, we establish a \emph{Faithfulness Theorem for vertex algebra modules}, showing that, in a natural class of vertex algebras, two complementary localizations together detect the original module. This provides a general mechanism for deducing global algebraic properties from local geometric realizations.

The underlying principle is simple. Suppose that a vertex algebra $V$ carries two compatible nonnegative conformal structures and admits two localizations related by Fourier transform. Under a nondegeneracy condition expressing a full contraction between the corresponding localization functions, we prove that a $V$-module vanishes whenever both of its localizations vanish. More generally, a morphism is injective or surjective whenever this is true after both localizations. This is the content of the \emph{Faithfulness Theorem} (Theorems~\ref{theorem:iso} and \ref{th:faithfulness-2}).

The strength of this principle is that many global questions about vertex algebras can then be reduced to geometric statements on suitable open subsets. In particular, when the localized vertex algebras are simple, the original vertex algebra is simple. Under Noetherian hypotheses, the same argument detects stabilization of submodules. In the geometric situation considered below, the localizations
become sheaves of chiral differential operators (\cite{MSV,GMS03}), whose local structure is explicit.

We apply this framework to vertex algebras obtained by BRST reduction from classical invariant theory. Let $G$ be a connected reductive algebraic group and $W_0$ a finite-dimensional $G$-module. Set \[ \widetilde X := W_0 /\!/ G = \operatorname{Spec}\mathcal O(W_0)^G. \] Suppose that $U\subset W_0$ is a $G$-stable open subset on which $G$ acts freely with closed orbits, and let \[ X := U/G \subset \widetilde X. \] Under the codimension hypothesis  \ref{A3}  below, \begin{equation} \mathcal O(X)=\mathcal O(W_0)^G. \label{eq:intro-coordinate-ring} \end{equation} Thus $X$ plays the role of a smooth geometric model for the possibly singular affine quotient $\widetilde X$. A direct chiralization of this picture encounters an obstruction: a sheaf of chiral differential operators on $X$ need not exist (\cite{GMS04}). 
We remove this obstruction by introducing a second $G$-module $W_1$ and the associated vector bundle \[ E=(U\times W_1)/G\longrightarrow X. \] The corresponding sheaf $\mathcal D_{\wedge E}^{\mathrm{ch}}$ (\cite{GMS03})
of chiral differential operators exists provided \begin{equation} \operatorname{tr}_{W_0}(xy)-\operatorname{tr}_{W_1}(xy) =\kappa_{\mathfrak g}(x,y), \qquad x,y\in\mathfrak g, \label{eq:intro-anomaly} \end{equation} together with the trace condition 
in  \eqref{eq:level-condition}. The former is precisely the anomaly-cancellation condition appearing in the gauge-theoretic interpretation.

Let \( \mathcal D_{\wedge F}^{\mathrm{ch}}(W_0) \) be the $\beta\gamma$--$bc$ system associated with 
the trivial bundle
\( F=W_0\times W_1\longrightarrow W_0\). We consider the relative BRST reduction \begin{equation} \mathbf{V}= H^{\frac{\infty}{2}+0} \left( \widehat{\mathfrak g}_{-\kappa_{\mathfrak g}}, \mathfrak g; \mathcal D_{\wedge F}^{\mathrm{ch}}(W_0) \right). \label{eq:intro-V} \end{equation} It carries two conformal structures exchanged by Fourier transform. Their interaction with the two natural geometric localizations is exactly what allows the Faithfulness Theorem to be applied.

Our first main application is that the BRST construction is globally controlled by its restriction to the free locus. We prove that the  relative BRST cohomology vanishes 
in nonzero degrees
and that $\mathbf{V}$ is simple (Theorem~\ref{Th:vanishing-and-simplicity}). Moreover, restriction gives an isomorphism 
\begin{equation} \mathbf V \xrightarrow{\sim} \Gamma\left(X,\mathcal D_{\wedge E}^{\mathrm{ch}}\right) \label{eq:intro-global-sections} \end{equation}
(Theorem~\ref{th:no-support}).
Thus the BRST reduction is precisely the algebra of global sections of a sheaf of chiral differential operators on the smooth quotient $X$. Equivalently, the BRST sheaf on $\widetilde X$ provides an extension of $\mathcal D_{\wedge E}^{\mathrm{ch}}$ across the singular complement, while the Faithfulness Theorem shows that no global information is lost by passing to the appropriate local geometric models.

This should be viewed as a chiral counterpart of the philosophy underlying localization results of Teleman \cite{Tele} 
and McGerty--Nevins \cite{MN}: global algebraic information is recovered from geometric localization, but here the relevant objects are vertex algebras and their modules rather than ordinary algebras of differential operators.

The same mechanism has particularly strong consequences for representation theory. 
The isomorphism~\eqref{eq:intro-global-sections} allows suitable $\mathbf{V}$-modules to be localized to modules over $\mathcal D_{\wedge E}^{\mathrm{ch}}$. The Faithfulness Theorem then gives an exact faithful functor \begin{equation} \mathbf V\text{-mod}_{L,L',\geq 0} \longrightarrow \left(\mathcal D_{\wedge E}^{\mathrm{ch}}\text{-mod}\right)^{\oplus 2}, \qquad M\longmapsto \left( \operatorname{Loc}(M), \operatorname{Loc}(\mf{F}^*M) \right). \label{eq:intro-faithful-functor} \end{equation} For finitely strongly generated modules, the top
 weight components of these localizations become coherent integrable connections. Pulling them back to $U$, extending across the codimension-two complement, and using the Euler action shows that these connections are algebraically trivial. Faithfulness then upgrades this local geometric statement to a global classification: 
\begin{equation} \mathbf V\on{-mod}^{\mathrm{fin}} \simeq \mathbf{Vect}_{\mathrm{fd}}
\label{eq:intro-module-category} \end{equation} 
 (Theorem~\ref{th:equivalence}).
 In particular, every finitely strongly generated module is a finite 
direct sum of copies of $\mathbf V$, and \[ \operatorname{Ext}^1_{\mathbf V}(\mathbf V,\mathbf V)=0. \] Thus, in this setting, localization does more than realize modules geometrically: it completely determines their finite representation theory.

The BRST vanishing theorem also yields
integral formula for the supercharacter of $\mathbf V$ (Theorem~\ref{th:character}).

These constructions have a direct interpretation in supersymmetric gauge theory. The vertex algebra $\mathbf{V}$ is the 
$A$-twisted\footnote{Costello and Gaiotto originally called this the $H$-twisted boundary VOA.}
boundary VOA associated with the three-dimensional $\mathcal N=4$ gauge theory determined by $G$ and $W_0$, with boundary matter specified by $W_1$ (\cite{CG}; see also \cite{Gar,Y}).
 The anomaly-cancellation condition is exactly~\eqref{eq:intro-anomaly}. Consequently, \eqref{eq:intro-global-sections} gives geometric realizations of a large class of boundary VOAs by global sections of CDOs, see also \cite{AKM,AFM,CSYM,CSYM2}.
 For suitable $W_0$ and $W_1$, the same vertex algebra occurs as the chiral algebra of a four-dimensional $\mathcal N=2$ superconformal field theory under the 4D/2D correspondence \cite{BLLPRR}. 

A central prediction of  these connections is that the associated variety of the vertex algebra should recover the Higgs branch
 (\cite[Conjecture 2]{BR}, \cite[Conjecture 1]{BF}). In the present setting this predicts \begin{equation} X_{\mathbf{V}} \simeq T^*W_0 /\!/\!/_{0} G. \label{eq:intro-higgs} \end{equation} Since the reduced Hamiltonian reduction has finitely many symplectic leaves \cite{Lo}, this identification also implies that $\mathbf{V}$ is quasi-lisse \cite{AK}.

Our second main application is to prove~\eqref{eq:intro-higgs} for the classical groups \[ G=SL_n,\qquad Sp_{2n},\qquad SO_n, \] with \[ W_0=\operatorname{Hom}(\mathbb C^M,\mathbb C^n), \qquad W_1=\operatorname{Hom}(\mathbb C^N,\mathbb C^n), \] 
replacing $\mathbb C^n$ by $\mathbb C^{2n}$ in the symplectic case.
The parameters satisfy 
the anomaly-cancellation condition. The precise relations are \[ M-N=2n \quad (SL_n),\qquad M-N=2n+2 \quad (Sp_{2n}),\qquad M-N=n-2 \quad (SO_n). \] For these families we identify the reduced $C_2$-algebra with the 
coordinate ring
of the 
corresponding Hamiltonian reduction and hence establish the predicted equality between associated varieties and Higgs branches.

The classical families also produce concrete affine vertex superalgebras. For $G=Sp_{2n}$, we obtain \[ \mathbf{V}\simeq L_{-2n}(\mathfrak{osp}_{2M|2N}), \] and determine its associated variety: \[ X_{L_{-2n}(\mathfrak{osp}_{2M|2N})} \simeq \overline{\mathbb O}_{[2^{2n},1^{\,2M-4n}]} \subset \mathfrak{so}_{2M}. \] 
The equivalence \eqref{eq:intro-module-category} for this family provides
a geometric proof of the semisimplicity of 
the Kazhdan--Lusztig  category \cite{AKFP,AFP}.
The supercharacter formula  shows that its homogeneous supercharacter depends only on $n$; in particular, it coincides with the homogeneous character of \[ L_{-2n}(\mathfrak{so}_{4n+4}). \] For $n=1$, this proves a conjecture of Adamovi\'c and Milas
(\cite[Conjecture 8.1]{AdMi}).

For $G=SO_n$, the group \[ \mathbb Z_2=O_n/SO_n \] acts naturally on $\mathbf{V}$, and we prove \[ \mathbf{V}^{\mathbb Z_2} \simeq L_n(\mathfrak{osp}_{2N|2M}). \] Its associated variety is \[ X_{L_n(\mathfrak{osp}_{2N|2M})} \simeq \overline{\mathbb O}_{[2^n,1^{\,2M-2n}]} \subset \mathfrak{sp}_{2M}, \] while $X_{\mathbf{V}}$ is the corresponding degree-two cover.

For $G=SL_n$, $\mathbf{V}$ is a conformal extension of an affine vertex superalgebra associated with $\mathfrak{gl}_{M|N}$ at level $-n$. When $N=0$, it specializes to the chiral algebra associated with four-dimensional $\mathcal N=2$ superconformal QCD with gauge group $SU(n)$ 
and $2n$ flavors. 
This is equivalently the class \(\mc S\) theory associated with \(SL_n\) and \(\mathbb P^1\) with two maximal and two minimal punctures; see \cite{G,GW} and \cite[Example 4.4]{T}. 
We obtain a complete presentation in this case (Theorem~\ref{thm:relation}),
which proves \cite[Conjecture 2]{BLLPRR}.
For $n=2$, \( \mathbf{V}\simeq L_{-2}(\mathfrak{so}_8), \) recovering the corresponding class $S$ chiral algebra
constructed in \cite{A18}.

\medskip

In forthcoming work \cite{ADS}, we apply the Faithfulness Theorem to chiral differential operators on the basic affine space.

\medskip
This paper is organized as follows. Section~\ref{sec:preliminaries} reviews vertex algebras, loop modules, and free field systems. Section~\ref{sec:localization} develops localization and sheaves of BRST reductions. Section~\ref{sec:faithfulness} proves the general Faithfulness Theorem. Section~\ref{sec:reduction} applies it to chiral Hamiltonian reduction, establishing BRST vanishing, simplicity, the global-sections isomorphism, the equivalence of module categories and its consequence for self-extensions, and the supercharacter formula. Section~\ref{sec:classical} treats \(SL_n\), \(Sp_{2n}\), and \(SO_n\), including their associated varieties and affine realizations.

\medskip
Unless explicitly stated otherwise, all algebraic groups are connected, reductive, and defined over \(\C\).
Throughout, we forget the parity of modules:  morphisms are ordinary linear maps commuting with all modes. The parity of the vertex algebra is retained in the supercommutator and Borcherds identities. We also retain the natural parity gradings on free field systems and BRST complexes when taking supercharacters.

\medskip
\noindent
\textbf{Acknowledgements}.
We thank Christopher Beem, Tudor Dimofte,
Andrea Ferrari,
Niklas Garner,
Ian Grojnowski,  Toshiro Kuwabara, Sven M\"{o}ller, Leonardo
Rastelli and
Wenbin Yan
for valuable discussions.
Tomoyuki Arakawa is partially supported by JSPS KAKENHI Grant Numbers
21H04993, 25K21659 and 26H019. 
Part of paper was written during his stay at SLMath 
for the program
``Representation Theory Under the Influence of Quantum Field Theory".
He thanks the institute for its excellent working environment.
Xuanzhong Dai was partially supported by JSPS KAKENHI Grant Numbers 23K19008 and 24KF0155. Part of paper was completed during his stays at the Research Institute for Mathematical Sciences, Kyoto University, and the Okinawa Institute of Science and Technology. He thanks both institutes for their hospitality and excellent working environments.
He presented some of the results of this paper at Representation Theory XIX, the RIMS Satellite Seminars “VOAs, Number Theory, and Combinatorics", and the 2025 International Congress of Chinese Mathematicians. He thanks the organizers for their invitations and the participants for helpful discussions.
Bailin Song is supported by NSFC No.12671042.

\noindent \textbf{AI Acknowledgements}.
AI was used to assist with proofreading and typesetting improvements in a draft of the manuscript. The authors are solely responsible for all mathematical results and any remaining errors.
\section{Vertex algebras, loop modules, and free fields}\label{sec:preliminaries}
\subsection{Vertex algebras}
A \emph{field} on a vector superspace \(M\) is a formal series \(a(z)=\sum_{n\in\Z}a_{(n)}z^{-n-1}\) with coefficients in \(\on{End}M\) such that \(a_{(n)}m=0\) for every \(m\in M\) and all sufficiently large \(n\). Throughout, brackets denote supercommutators, and parity is understood for homogeneous elements. Two fields are \emph{mutually local} if \((z-w)^N[a(z),b(w)]=0\) for some \(N\geq0\). In this case,
\begin{equation}\label{eq:ope}
[a(z),b(w)]=\sum_{j=0}^{N-1}c_j(w)\frac{\partial_w^j}{j!}\delta(z-w),
\qquad \delta(z-w)=\sum_{n\in\Z}z^{-n-1}w^n,
\end{equation}
for suitable fields \(c_j(w)\). This is also written as the operator product expansion (OPE) \(a(z)b(w)\sim\sum_{j\geq0}c_j(w)(z-w)^{-j-1}\). The \emph{\(n\)-th product} is
\[
(a_{(n)}b)(w)=\on{Res}_z\bigl(
i_{z,w}(z-w)^n a(z)b(w)
-(-1)^{p(a)p(b)}i_{w,z}(z-w)^n b(w)a(z)\bigr),
\]
where \(i_{z,w}\) and \(i_{w,z}\) denote expansion in \(|z|>|w|\) and \(|w|>|z|\), respectively. These products are 
mutually local
with \(a\) and \(b\). The \((-1)\)-st product is the \emph{normally ordered product}, denoted by \(:a(w)b(w):\).

We use the following definition \cite{DeSole-Kac}.
\begin{definition}\label{Def:vertex_algebra-v2}
A \emph{vertex algebra} is a vector superspace \(V\) equipped with a nonzero even vector \(\vac\), an even endomorphism \(T\), and a set \(\mc S\) of mutually local homogeneous fields such that:
\begin{enumerate}
\item\label{Def:vertex_algebra-v2-i}
\(T\vac=0\), \([T,a(z)]=\partial_z a(z)\);
\item\label{Def:vertex_algebra-v2-ii}
\(V\) is spanned by the vectors
\begin{equation}\label{eq:spanning}
a^{i_1}_{(n_1)}\cdots a^{i_m}_{(n_m)}\vac,
\end{equation}
\end{enumerate}
The vector \(\vac\), operator \(T\), and fields in \(\mc S\) are called the \emph{vacuum}, \emph{translation operator}, and \emph{generating fields}, respectively.
We have \(a(z)\vac\in V[[z]]\) for all \(a(z)\in\mc S\) (\cite[Lemma 1.2]{DeSole-Kac}).
\end{definition}

Given mutually local fields \(\mc S\) on \(M\), let \(\langle\mc S\rangle_M\) be their closure, together with \(\on{id}_M\), under linear combinations and all \(n\)-th products. By \cite{Li96}, this is a vertex algebra with vacuum \(\on{id}_M\), translation \(\partial_z\), and generating fields \(Y(a(z),\xi)=\sum_{n\in\Z}a(z)_{(n)}\xi^{-n-1}\), acting by \(n\)-th products. For a vertex algebra \(V\) generated by \(\mc S\), the assignment \(a_{(-1)}\vac\mapsto a(z)\) extends to the \emph{state--field correspondence}
\begin{equation}\label{eq:state-field}
V\isomap\langle\mc S\rangle_V,\qquad a\longmapsto a(z),
\end{equation}
whose inverse sends \(a(z)\) to \(a(z)\vac|_{z=0}\). We freely identify states with their fields.

A \emph{\(V\)-module} is a vector space \(M\) with a linear assignment \(a\mapsto a^M(z)\) of truncated fields satisfying the vacuum identity and the module Borcherds identity. We write \(a^M(z)=\sum_n a^M_{(n)}z^{-n-1}\), omitting the superscript when unambiguous. For homogeneous \(a,b\in V\), the iterate and commutator identities are
\begin{align}
(a_{(n)}b)^M_{(m)}
&=\sum_{i\geq0}(-1)^i\binom ni
\bigl(a^M_{(n-i)}b^M_{(m+i)}
-(-1)^{p(a)p(b)+n}b^M_{(m+n-i)}a^M_{(i)}\bigr),
\label{identity:Bor}\\
[a^M_{(n)},b^M_{(m)}]
&=\sum_{i\geq0}\binom ni(a_{(i)}b)^M_{(m+n-i)}.
\label{identity:Com}
\end{align}

A vertex algebra is \emph{graded} if \(V=\bigoplus_{\Delta\in\C}V_\Delta\) and
\begin{equation}\label{eq:grading}
a_{(n)}b\in V_{\Delta_a+\Delta_b-n-1}
\qquad(a\in V_{\Delta_a},\ b\in V_{\Delta_b}).
\end{equation}
A \emph{graded module} has a decomposition \(M=\bigoplus_{d\in\C}M_d\) satisfying
\begin{equation}\label{eq:grading-module}
a_{(n)}M_d\subset M_{d+\Delta_a-n-1}
\qquad(a\in V_{\Delta_a}).
\end{equation}
It is \emph{positively graded} if the real parts of its nonzero degrees are bounded below. We denote the category of such modules by \(V\on{-mod}_{\geq0}\).

An even vector \(L\in V\) is a \emph{conformal vector} if \(L(z)=\sum_n L_nz^{-n-2}\) satisfies \(L_{-1}=T\) and
\begin{equation}\label{Virasoro condition}
[L_m,L_n]=(m-n)L_{m+n}
+\frac{c}{12}(m^3-m)\delta_{m+n,0}\on{id}_V,
\end{equation}
where \(c\in\C\) is the \emph{central charge}. We call \((V,L)\) a \emph{conformal vertex algebra} if \(L_0\) is semisimple; its eigenspaces give the conformal grading. For conformal modules we require only local finiteness of \(L_0\) and use generalized eigenspaces in \eqref{eq:grading-module}. It is \emph{conical} if \(V=\bigoplus_{\Delta\in\mathbb Q_{\geq0}}V_\Delta\) and \(V_0=\C\vac\).

A vertex algebra is \emph{commutative} if \(a_{(n)}b=0\) for all \(a,b\in V\) and \(n\geq0\), equivalently if all fields supercommute. Such vertex algebras are precisely unital supercommutative algebras with an even derivation: multiplication is the \((-1)\)-st product and the derivation is \(T\).

For an affine scheme \(X=\on{Spec}A\), the coordinate ring \(J_\infty A:=\mc O(J_\infty X)\) of its arc space is naturally a differential algebra, hence a commutative vertex algebra. Following Borisov \cite{B}, we call a vertex \(J_\infty A\)-module an \emph{\(A\)-loop module}. This differs from an ordinary \(J_\infty A\)-module: in general, \((ab)^M_{(-1)}\ne a^M_{(-1)}b^M_{(-1)}\). An \(A\)-loop module is an ordinary \(J_\infty A\)-module under the \((-1)\)-st products if
\begin{equation}\label{eq:usual-module}
a^M_{(n)}=0\qquad(a\in A,\ n\geq0).
\end{equation}
We grade \(J_\infty A\) by assigning weight zero to \(A\).

An \emph{\(A\)-loop module filtration} is an exhaustive increasing filtration \(0=F_{-1}M\subset F_0M\subset\cdots\) by \(A\)-loop submodules such that \(a^M_{(n)}F_pM\subset F_{p-1}M\) for \(a\in A\) and \(n\geq0\). Its associated graded is an ordinary \(J_\infty A\)-module. Every positively graded \(A\)-loop module admits such a filtration: take the loop submodules generated by successively larger bounded-weight subspaces and use \eqref{eq:grading-module}.

Let \(V\) be a \(\mathbb Q_{\geq0}\)-graded vertex algebra. By \eqref{eq:grading}, \(a_{(n)}b=0\) for \(a,b\in V_0\) and \(n\geq0\). Thus \(V_0\) is a unital supercommutative algebra under the \((-1)\)-st product, and the vertex subalgebra it generates is commutative. For a unital commutative \(\C\)-algebra \(A\), a \emph{vertex \(A\)-algebra} is such a \(V\) equipped with a ring homomorphism \(A\to V_0\). This extends uniquely to a vertex algebra homomorphism
\begin{equation}\label{eq:vertex-A-algebra}
J_\infty A\longrightarrow V,
\end{equation}
so every \(V\)-module is naturally an \(A\)-loop module.

For a vertex \(A\)-algebra \(V\), Li's increasing filtration \cite{Li04} is
\[
F_pV=\on{span}_{\C}\bigl\{
a^1_{(-n_1-1)}\cdots a^r_{(-n_r-1)}\vac
\mid a^i\in V_{\Delta_i},\ n_i\geq0,\ \textstyle\sum_i\Delta_i\leq p
\bigr\},
\]
with \(F_pV=0\) for \(p<0\). For rational conformal weights, filtration indices are rational and associated graded spaces are formed degreewise. This filtration is exhaustive and satisfies
\begin{equation}\label{eq:condition-filtration}
(F_pV)_{(n)}F_qV\subset
\begin{cases}
F_{p+q}V,&n<0,\\
F_{p+q-1}V,&n\geq0.
\end{cases}
\end{equation}
The subspace \(F_0V\) is the commutative vertex subalgebra generated by \(V_0\), and each \(F_pV\) is an \(A\)-loop submodule. The filtration is finite on each conformal weight space: writing \(F_pV_\Delta=F_pV\cap V_\Delta\), we have \(V_\Delta=F_\Delta V_\Delta\), and the filtration jumps only at \(\Delta-\mathbb Z_{\geq0}\). It is equivalent, by degreewise reindexing, to the decreasing Li filtration \cite{Li}; see \cite[Proposition 2.6.1]{A12}. Unless otherwise stated, \(F_\bullet V\) denotes this increasing filtration
for a vertex algebra $V$.

The associated graded
\[
\on{gr}_FV=\bigoplus_{\Delta}\bigoplus_{j\geq0}
F_{\Delta-j}V_\Delta/F_{\Delta-j-1}V_\Delta
\]
is a vertex Poisson algebra, and its diagonal part
\[
R_V=\bigoplus_\Delta V_\Delta/F_{\Delta-1}V_\Delta
\subset\on{gr}_FV
\]
is a Poisson superalgebra with operations
\(\bar a\bar b=\overline{a_{(-1)}b}\) and
\(\{\bar a,\bar b\}=\overline{a_{(0)}b}\).
It is naturally isomorphic to Zhu's \(C_2\)-algebra \(V/C_2(V)\) \cite{Li,Z}, where \(C_2(V)=\on{span}_{\C}\{a_{(-2)}b\mid a,b\in V\}\). Moreover, \(\on{gr}_FV\) is generated by \(R_V\) as a differential algebra \cite[Lemma 4.1]{Li}, giving a canonical surjection of vertex Poisson algebras
\begin{equation}\label{eq:Li-surjection}
J_\infty R_V\twoheadrightarrow\on{gr}_FV.
\end{equation}
We call \(V\) \emph{finitely strongly generated} if \(R_V\) is finitely generated. The map \(A\to V_0\) induces \(A\to R_V\) with Poisson-commutative image. The \emph{associated scheme} and \emph{associated variety} are
\[
\tilde X_V:=\on{Spec}R_V,\qquad
X_V:=(\tilde X_V)_{\mathrm{red}};
\]
in particular, \(\tilde X_V\) is a scheme over \(\on{Spec}A\).

For a \(\mathbb Z_{\geq0}\)-graded vertex algebra, Zhu's algebra is \(\Zhu(V)=V/O(V)\), where \(O(V)\) is spanned by the elements \(a\circ b\) below, and its multiplication is induced by \(a*b\):
\[
a\circ b=\sum_{i\geq0}\binom{\Delta_a}{i}a_{(i-2)}b,
\qquad
a*b=\sum_{i\geq0}\binom{\Delta_a}{i}a_{(i-1)}b
\qquad(a\in V_{\Delta_a}).
\]
For homogeneous \(a,b\), we have
\[
a*b-(-1)^{p(a)p(b)}b*a
\equiv\sum_{i\geq0}\binom{\Delta_a-1}{i}a_{(i)}b
\pmod{O(V)}.
\]
Thus the image of \(A\) in \(\Zhu(V)\) is commutative. The images \(F_p\Zhu(V)\) of \(F_pV\) form an exhaustive filtration satisfying
\begin{equation}\label{eq:Zhu-commutator}
F_p\Zhu(V)F_q\Zhu(V)\subset F_{p+q}\Zhu(V),
\qquad
[F_p\Zhu(V),F_q\Zhu(V)]\subset F_{p+q-1}\Zhu(V).
\end{equation}
Consequently, \(\on{gr}_F\Zhu(V)\) is a Poisson superalgebra, and there is a canonical surjection \(R_V\twoheadrightarrow\on{gr}_F\Zhu(V)\).

\subsection{\texorpdfstring{Example: the \(\beta\gamma bc\) system}{Example: the beta-gamma bc system}}
\label{subsection:betagamma bc-system}
Let \(W_0,W_1\) be finite-dimensional complex vector spaces and \(F=W_0\times W_1\to W_0\) the trivial vector bundle from the Introduction. The \(\beta\gamma\) system \(\mc S(W_0)\) is generated by even fields \(\beta^v,\gamma^{v^*}\), linear in \(v\in W_0\) and \(v^*\in W_0^*\); the \(bc\) system \(\mc E(W_1)\) is generated by odd fields \(b^w,c^{w^*}\), linear in \(w\in W_1\) and \(w^*\in W_1^*\). Their defining OPEs are
\[
\beta^v(z)\gamma^{v^*}(w)\sim\frac{\langle v,v^*\rangle}{z-w},
\qquad
b^w(z)c^{w^*}(w)\sim\frac{\langle w,w^*\rangle}{z-w},
\]
with all remaining OPEs determined by locality or regularity. Their tensor product is
\begin{equation}\label{eq:cdo-wedgeF}
\mc D^{\mathrm{ch}}_{\wedge F}(W_0)
:=\mc S(W_0)\otimes\mc E(W_1).
\end{equation}
For bases \(\{v_i\}\) of \(W_0\) and \(\{w_j\}\) of \(W_1\), with dual bases \(\{v_i^*\}\) and \(\{w_j^*\}\), the conformal vector
\begin{equation}\label{eq:conformal-vector-of-beta-gamma}
L=\sum_i:\beta^{v_i}\partial\gamma^{v_i^*}:
-\sum_j:b^{w_j}\partial c^{w_j^*}:
\end{equation}
gives a \(\mathbb Z_{\geq0}\)-grading in which \(\beta,b\) have weight one and \(\gamma,c\) have weight zero. Hence
\[
\mc D^{\mathrm{ch}}_{\wedge F}(W_0)_0
\cong\on{Sym}(W_0^*)\otimes\bigwedge W_1^*
=\mc O(W_0)\otimes\bigwedge W_1^*,
\]
so \(\cdo{\wedge F}(W_0)\) is a vertex \(\mc O(W_0)\)-algebra. We have
\begin{align}
R_{\cdo{\wedge F}(W_0)}
&\cong\mc O(\wedge T^*F)
\cong\mc O(T^*W_0)\otimes\bigwedge W_1\otimes\bigwedge W_1^*,
\label{eq:R-of-cdo}\\
\Zhu(\cdo{\wedge F}(W_0))
&\cong\on{Weyl}(W_0)\otimes\on{Cl}(W_1).
\end{align}
Here \(\mc O(\wedge T^*F)\) denotes the coordinate superalgebra of \(T^*W_0\times\Pi T^*W_1\), while \(\on{Weyl}(W_0)\) and \(\on{Cl}(W_1)\) are the Weyl and Clifford algebras associated with \(T^*W_0\) and \(\Pi T^*W_1\), respectively. In particular, \(\tilde X_{\cdo{\wedge F}(W_0)}\cong T^*W_0\times\Pi T^*W_1\).

\section{Localization and BRST reduction}\label{sec:localization}
Localization of vertex algebras and their modules was introduced in \cite{MSV} and developed in \cite{B}. We review the construction in a slightly more general setting. Throughout, \(A\) is a unital commutative \(\C\)-algebra, \(S\subset A\) is multiplicative, and  module homomorphisms preserve the grading.

\subsection{\texorpdfstring{Localization of \(A\)-loop modules}{Localization of A-loop modules}}
\begin{lemma}\label{lem:inverse}
If \(M\in A\on{-loop-mod}_{\geq0}\) and \(s\in A^\times\), then \(s^M_{(-1)}\) is invertible.
\end{lemma}
\begin{proof}
Choose an \(A\)-loop module filtration of \(M\). On \(\on{gr}_FM\), multiplication by \(s\) has inverse multiplication by \(s^{-1}\). The assertion follows by induction on the filtration degree, since the filtration is exhaustive and bounded below.
\end{proof}

Write \(A_S\) for the ordinary localization; then \(J_\infty(A_S)\cong(J_\infty A)_S\).
\begin{proposition}\label{pro:localization-of-loop-modules}
Every \(M\in A\on{-loop-mod}_{\geq0}\) admits a localization \(\varphi\colon M\to M_S\), with \(M_S\in A_S\on{-loop-mod}_{\geq0}\), characterized by the following property: every \(A\)-loop module homomorphism \(\phi\colon M\to N\), with \(N\in A_S\on{-loop-mod}_{\geq0}\), factors uniquely through an \(A_S\)-loop module homomorphism \(M_S\to N\).
\end{proposition}
\begin{proof}
All modes of elements of \(A\) commute. Let \(A_M\subset\on{End}M\) be generated by \(a^M_{(-1)}\), \(a\in A\), and let \(S_M\) be the multiplicative subset generated by \(s^M_{(-1)}\), \(s\in S\). Define \(M_S:=S_M^{-1}M\). Distinguishing \(s\in S\) from its operator \(D_s:=s^M_{(-1)}\), we extend \(a^M(z)\) coefficientwise to \(M_S\) and define
\begin{equation}\label{eq:def-of-field-of-of-inverse}
(s^{-1})^{M_S}(z)
=\sum_{l\geq0}(-1)^lD_s^{-l-1}
\left(\sum_{k\ne-1}s^M_{(k)}z^{-k-1}\right)^l.
\end{equation}
Positive grading and commutativity make this expansion coefficientwise finite on each vector and ensure the field truncation property. These inverse fields commute with all \(A\)-fields and satisfy \(s(z)(s^{-1})(z)=\on{id}\). Hence the map from \(J_\infty A\) to the resulting commutative vertex algebra extends to \(J_\infty(A_S)\), giving the required loop module structure. The grading extends because every \(D_s\) has degree zero.

For \(\phi\colon M\to N\) in the statement, Lemma~\ref{lem:inverse} allows each \(D_s\) to be inverted on \(N\). Thus \(\phi\) extends uniquely to \(S_M^{-1}M\); formula \eqref{eq:def-of-field-of-of-inverse} shows that this extension intertwines the inverse fields and is an \(A_S\)-loop module homomorphism.
\end{proof}

In general, \((s^{-1})^{M_S}_{(-1)}m\ne D_s^{-1}m\); equality holds if \(s^M_{(n)}m=0\) for all \(n\geq0\). For an ordinary \(J_\infty A\)-module, this construction is ordinary localization. Moreover, \(\ker\varphi\) consists of the vectors killed by some product \(D_{s_1}\cdots D_{s_r}\).

\begin{lemma}\label{lem:exact}
Localization defines an exact functor \(A\on{-loop-mod}_{\geq0}\to A_S\on{-loop-mod}_{\geq0}\).
\end{lemma}
\begin{proof}
This follows from exactness of ordinary localization over the commutative algebra generated by the \((-1)\)-st modes of \(A\).
\end{proof}

An \(A\)-loop module filtration localizes to an \(A_S\)-loop module filtration by
\begin{equation}\label{eq:localization-of-filtration}
F_pM_S:=(F_pM)_S.
\end{equation}
For \(m\in F_pM_S\), formula \eqref{eq:def-of-field-of-of-inverse} gives \((s^{-1})_{(-1)}m\equiv D_s^{-1}m\pmod{F_{p-1}M_S}\). Exactness now yields the following.
\begin{lemma}\label{lemma:filtration-commutes-with-localization}
There is a natural isomorphism \(\on{gr}_F(M_S)\cong(\on{gr}_FM)_S\), where the right-hand side is ordinary localization.
\end{lemma}
\begin{proof}
Localize \(0\to F_{p-1}M\to F_pM\to F_pM/F_{p-1}M\to0\) and take the direct sum over \(p\).
\end{proof}

\begin{lemma}\label{lem:spanning-set-of-localization}
The space \(M_S\) is spanned by \((s^{-1})_{(-1)}m\), with \(s\in S\) and \(m\in M\).
\end{lemma}
\begin{proof}
Apply Lemma~\ref{lemma:filtration-commutes-with-localization} and induction on the filtration degree.
\end{proof}

\begin{lemma}\label{lemm:elements-of-localization}
For every \(m\in M_S\), some product \(D_{s_1}\cdots D_{s_r}\), with \(s_i\in S\), sends \(m\) into the image of \(M\).
\end{lemma}
\begin{proof}
This is immediate from \(M_S=S_M^{-1}M\).
\end{proof}

For \(f\in A\), write \(M_f\) for localization at \(\{f^r\mid r\geq0\}\).
\begin{lemma}\label{lem:in-localization}
Let \(M\in A\on{-loop-mod}_{\geq0}\).
\begin{enumerate}
\item For \(m\in M\), the following are equivalent: \(m\mapsto0\) in \(M_f\); \((f_{(-1)})^Nm=0\) for some \(N\); and \((f^N)_{(-1)}m=0\) for some \(N\).
\item If \(m\in F_pM_f\), then each of \((f_{(-1)})^Nm\) and \((f^N)_{(-1)}m\) lies in the image of \(F_pM\) for all sufficiently large \(N\).
\end{enumerate}
\end{lemma}
\begin{proof}
For homogeneous \(m\), the identity \(f^N(z)=f(z)^N\) gives
\begin{equation}\label{eq:identify-f^N}
\begin{aligned}
(f^N)_{(-1)}m={}&(f_{(-1)})^Nm\\
&+\sum_{\substack{k,l\geq1\\k+l\leq N}}
\ \sum_{\substack{i_1,\ldots,i_l,j_1,\ldots,j_k\geq1\\\sum i_s=\sum j_t}}
C_{N;\boldsymbol i,\boldsymbol j}
\left(\prod_{s=1}^l f_{(-i_s-1)}\right)
(f_{(-1)})^{N-k-l}
\left(\prod_{t=1}^k f_{(j_t-1)}\right)m .
\end{aligned}
\end{equation}
The lower bound on weights bounds \(\sum j_t\), and hence \(k,l\), by an integer \(D\) depending on the degree of \(m\). All modes commute, and every correction term contains a positive-degree-lowering mode. Such correction terms therefore generate a nilpotent ideal on the cyclic module generated by \(m\). It follows that inverting \(f_{(-1)}\) also inverts every \((f^r)_{(-1)}\); thus \(M_f\) can be computed by inverting \(f_{(-1)}\) alone, as in \cite[Remark 2.4]{B}. This proves the equivalence of the first two conditions in (1), and 
establishes (2) for operator powers.

If \((f_{(-1)})^{N_0}m=0\), then \eqref{eq:identify-f^N} vanishes for \(N\geq N_0+2D\), proving the third condition in (1). Its converse follows from Lemma~\ref{lem:inverse} applied after localization. The same expansion proves (2) for \((f^N)_{(-1)}\): once \((f_{(-1)})^{N_0}m\) lies in the image of \(F_pM\), every potentially nonzero summand does so for \(N\geq N_0+2D\).
\end{proof}

\subsection{Quasi-loop-coherent sheaves}
Set \(X=\on{Spec}A\). If \(X_f\subset X_g\), then \(g\) is invertible in \(A_f\), so the universal property gives compatible restriction maps \(M_g\to M_f\). These define a presheaf \(\Loc{X}{M}\) on the basis of principal opens, with \(\Loc{X}{M}(X_f)=M_f\). Likewise, \(F_p\Loc{X}{M}(X_f)=F_pM_f\).
\begin{lemma}\label{lem:Loc-is-a-sheaf}
The presheaves \(\Loc{X}{M}\) and \(F_p\Loc{X}{M}\) are sheaves.
\end{lemma}
\begin{proof}
By Lemma~\ref{lemma:filtration-commutes-with-localization}, their successive filtration quotients are quasi-coherent sheaves. Exactness of the corresponding \v Cech complexes on finite principal-open covers proves the assertion by induction on \(p\) and passage to the exhaustive union.
\end{proof}

A sheaf \(\mc M\) of positively graded complex vector spaces on a variety \(Y\) is \emph{quasi-loop-coherent}, or \emph{quasi-loco} \cite{B}, if on each affine open \(U\subset Y\) it is the localization sheaf of an \(\mc O(U)\)-loop module, compatibly with restriction. Write \(\on{QLCoh}(Y)\) for their category. The grading filtration generated locally by bounded-weight subspaces gives an exhaustive filtration by quasi-loco subsheaves \(F_p\mc M\), whose sections on affine opens form loop module filtrations. The quotients \(F_p\mc M_d/F_{p-1}\mc M_d\) are quasi-coherent \(\mc O_Y\)-modules. For \(X=\on{Spec}A\), localization and global sections give an equivalence
\begin{align}\label{eq:quasi-loco-equivalence}
A\on{-loop-mod}_{\geq0}\simeq\on{QLCoh}(X),
\qquad M\longmapsto\Loc{X}{M}.
\end{align}
\subsection{\texorpdfstring{Localization of vertex \(A\)-algebras and their modules}{Localization of vertex A-algebras and their modules}}
\label{subsection:localization-of-vertex-algebra}
Let \(V\) be a vertex \(A\)-algebra and \(M\in V\on{-mod}_{\geq0}\). Define \(V_S\) and \(M_S\) initially as localizations of \(A\)-loop modules.
\begin{lemma}\label{lem:kernel-is-a-submodule}
The kernel of \(M\to M_S\) is a \(V\)-submodule.
\end{lemma}
\begin{proof}
Put \(K=\ker(M\to M_S)\) and argue by induction on the Li filtration of \(a\in V\). The assertion is clear for \(a\in F_0V\). For \(s\in S\), the commutator formula gives
\[
[D_s,a_{(n)}]=\sum_{j\geq0}(-1)^j(s_{(j)}a)_{(n-j-1)},
\qquad s_{(j)}F_pV\subset F_{p-1}V\quad(j\geq0).
\]
Thus \(D_sa_{(n)}m\equiv a_{(n)}D_sm\pmod K\) for \(m\in K\). Choose a product \(D_{s_1}\cdots D_{s_r}\) killing \(m\). Commuting this product past \(a_{(n)}\) shows that it sends \(a_{(n)}m\) into \(K\). Since \(K\) is saturated with respect to \(S_M\), we obtain \(a_{(n)}m\in K\).
\end{proof}

\begin{lemma}\label{lem:the-action-on-the-localization}
The \(V\)-action on \(M\) extends to \(M_S\). For \(u\in S_M\), it is given by
\begin{equation}\label{eq:the-action-on-the-localization}
a^{M_S}(z)(u^{-1}m)
=\sum_{k\geq0}u^{-k-1}\on{ad}_u^k(a^M(z))m,
\qquad \on{ad}_u(B):=[u,B].
\end{equation}
\end{lemma}
\begin{proof}
By Lemma~\ref{lem:kernel-is-a-submodule}, we may replace \(M\) by its image in \(M_S\). Commuting with \(D_s\) lowers the Li filtration, so the sum in \eqref{eq:the-action-on-the-localization} is finite for each \(a\). The identity
\begin{equation}\label{eq:action}
a(z)u^{-1}m=u^{-1}a(z)m+u^{-1}[u,a(z)]u^{-1}m
\end{equation}
shows compatibility with changes of denominator. It also proves, by induction on the filtration degrees, that products and supercommutators of modes extend compatibly with localization. In particular, the vacuum and Borcherds identities on \(M\) remain valid on \(M_S\). Thus the extended fields define a \(V\)-module structure.
\end{proof}

\begin{prop}\label{pro:locality-of-inverse}
For \(s\in S\) and 
a homogeneous element
\(a\in V\)
of weight \(\Delta_a\), 
we have
\begin{align*}
(z-w)^j[a^{M_S}(z), (s^{-1})^{M_S}(w)]=0\quad \text{for }j\geq \Delta_a.
\end{align*}
In particular,
the fields  \(a^{M_S}(z)\) and \((s^{-1})^{M_S}(z)\)  are mutually local.
\end{prop}
\begin{proof}
Since \((1-s)^k\in V_0\), the grading gives
\(a_{(j)}(1-s)^k=0\) for \(j\geq \Delta_a\).
Thus,
\begin{align}\label{eq:uniform-bound}
(z-w)^r
[a(z), ((1-s)^k)(w)]=0
\text{ for all }r\geq \Delta_a\text{ and }k\geq0.
\end{align}

Introduce a variable \(\hbar\) of weight zero and set
\begin{align*}
f_\hbar=1-\hbar+\hbar s\in A[\hbar].
\end{align*}
For each \(N\geq1\), the field \(f_\hbar(z)\) is invertible
on the \(A[\hbar]\)-loop module
\(M_S[\hbar]/(\hbar^N)\), with inverse
\begin{equation}\label{eq:inverse-in-hadic}
\sum_{k=0}^{N-1}\hbar^k((1-s)^k)(z).
\end{equation}
Hence \(M_S[\hbar]/(\hbar^N)\) is naturally an
\(A[\hbar]_{f_\hbar}\)-loop module.
The universal property gives a unique homomorphism
\[
\psi_N\colon  M_S[\hbar]_{f_{\hbar}}\longrightarrow M_S[\hbar]/(\hbar^N),
\qquad
\iota_N=\psi_N\circ\varphi,
\]
where \(\varphi\colon M_S[\hbar]\to M_S[\hbar]_{f_{\hbar}}\) is the
localization map and \(\iota_N\colon M_S[\hbar]\to M_S[\hbar]/(\hbar^N)\)
is the natural projection.

The maps \(\psi_N\) are compatible with the projections
\(M_S[\hbar]/(\hbar^{N+1})\to M_S[\hbar]/(\hbar^N)\), and therefore induce a linear map
\[
\psi\colon M_S[\hbar]_{f_{\hbar}}\longrightarrow
\varprojlim_N M_S[\hbar]/(\hbar^N)=M_S[[\hbar]]
\]
such that \(\psi\circ\varphi=\iota\), where
\(\iota\colon M_S[\hbar]\hookrightarrow M_S[[\hbar]]\)
is the natural embedding.
By Lemma~\ref{lem:in-localization},
the injectivity of
\(\iota=\psi\circ\varphi\)
implies that  \(\psi\) is injective.

The spaces \(M_S[\hbar]\) and \(M_S[\hbar]/(\hbar^N)\) are 
\(V[\hbar]\)-modules, and
Lemma~\ref{lem:the-action-on-the-localization}
extends this structure to \(M_S[\hbar]_{f_\hbar}\).
Formula~\eqref{eq:the-action-on-the-localization}
shows that \(\varphi\), \(\iota_N\), and \(\psi_N\)
are \(V[\hbar]\)-module homomorphisms.

By \eqref{eq:uniform-bound} and
\eqref{eq:inverse-in-hadic}, we have
\[
(z-w)^r[a(z),(f_\hbar^{-1})(w)]=0\quad \text{for }r\geq \Delta_a,
\]
on every \(M_S[\hbar]/(\hbar^N)\).
Since \(\psi\) is injective, the same identity holds
on \(M_S[\hbar]_{f_\hbar}\).

Finally, the exactness of loop module localization gives
\begin{equation}\label{eq:evaluation}
M_S[\hbar]_{f_\hbar}/(\hbar-1)M_S[\hbar]_{f_\hbar}
\cong
\bigl(M_S[\hbar]/(\hbar-1)M_S[\hbar]\bigr)_s
\cong (M_S)_s\cong M_S.
\end{equation}
This is an isomorphism of loop modules over
\(A[\hbar]_{f_\hbar}/(\hbar-1)\cong A_s\).
Moreover, the specialization map intertwines the
\(V\)-fields by
\eqref{eq:the-action-on-the-localization}
and sends the inverse field \(f_\hbar^{-1}(w)\)
to \((s^{-1})(w)\).
Specializing the preceding locality identity at
\(\hbar=1\) therefore yields
\[
(z-w)^r[a^{M_S}(z),(s^{-1})^{M_S}(w)]=0
\quad \text{for }r\geq \Delta_a,
\]
as required.
\end{proof}

\begin{theorem}\label{Th:localization}
Let \(V\) be a vertex \(A\)-algebra and \(S\subset A\) multiplicative.
\begin{enumerate}
\item The loop module localization \(V_S\) is naturally a vertex \(A_S\)-algebra, and \(\iota\colon V\to V_S\) is a graded vertex algebra homomorphism. Every graded map \(\phi\colon V\to W\) of vertex algebras with \(W\) nonnegatively graded and \(\phi(S)\subset W_0^\times\) factors uniquely through a graded vertex algebra homomorphism \(V_S\to W\).
\item For \(M\in V\on{-mod}_{\geq0}\), \(M_S\) is a \(V_S\)-module with a \(V\)-linear map \(\eta_M\colon M\to M_S\). Every \(V\)-module homomorphism \(M\to N\), with \(N\in V_S\on{-mod}_{\geq0}\), factors uniquely through a \(V_S\)-module homomorphism \(M_S\to N\).
\end{enumerate}
\end{theorem}
\begin{proof}
(1) By Proposition~\ref{pro:locality-of-inverse}, the fields \(a(z)\), \(a\in V\), and \((s^{-1})(z)\), \(s\in S\), generate a vertex algebra of fields on \(V_S\). The translation operator extends by
\[
T(D_s^{-1}a)=D_s^{-1}Ta+D_s^{-1}[D_s,T]D_s^{-1}a,
\]
and satisfies translation covariance for both types of generators. 
Lemma~\ref{lem:spanning-set-of-localization} gives the spanning property. Reconstruction therefore equips \(V_S\) with a vertex \(A_S\)-algebra structure. Given \(\phi\colon V\to W\), Proposition~\ref{pro:localization-of-loop-modules} provides its unique loop module extension. Formula \eqref{eq:the-action-on-the-localization} shows that it intertwines the \(V\)-fields, and \eqref{eq:def-of-field-of-of-inverse} shows the same for inverse fields. Hence it is a vertex algebra homomorphism.

(2) The extended \(V\)-fields and inverse fields on \(M_S\) are mutually local and give the corresponding representation of \(V_S\). The universal property follows by extending \(M\to N\) as a loop module map and using the same two formulas to check \(V_S\)-linearity.
\end{proof}

Lemma~\ref{lem:spanning-set-of-localization} shows that Li's increasing filtration commutes with localization:
\begin{equation}\label{eq:Li-filtration}
F_p(V_S)=(F_pV)_S.
\end{equation}
Combining this with Lemma~\ref{lemma:filtration-commutes-with-localization} gives the following.
\begin{lemma}\label{eq:filtration-commutes-with-localization2}
There are natural isomorphisms \(\on{gr}_FV_S\cong(\on{gr}_FV)_S\) and \(R_{V_S}\cong(R_V)_S\).
\end{lemma}

\subsection{Sheaves of vertex algebras}
A \emph{sheaf of vertex algebras} on a variety \(Y\) is a quasi-loco sheaf \(\mc V\) such that each \(\mc V(U)\) is a vertex \(\mc O_Y(U)\)-algebra and all restriction maps are compatible vertex algebra homomorphisms. A \emph{\(\mc V\)-module} is a quasi-loco sheaf \(\mc M\) whose sections \(\mc M(U)\) are \(\mc V(U)\)-modules, compatibly with restriction. Write \(\mc V\on{-mod}\) for their category.

For \(X=\on{Spec}A\), Theorem~\ref{Th:localization} makes \(\Loc{X}{V}\) a sheaf of vertex algebras and gives an exact functor
\begin{align}\label{eq:localization-functor-for-affine}
V\on{-mod}_{\geq0}\longrightarrow\Loc{X}{V}\on{-mod},
\qquad M\longmapsto\Loc{X}{M},
\end{align}
with right adjoint the global sections functor \(\Gamma(X,-)\).

\subsection{Localization commutes with BRST reduction}
\label{subsection:Localization commutes with BRST reduction}
Let \(V\) be a vertex \(A\)-algebra, \(M\in V\on{-mod}_{\geq0}\), and \(Q\in V\) odd with \(Q(z)Q(w)\sim0\). Then \(d:=Q_{(0)}\) satisfies \(d^2=0\). Assume that \(Q_{(n)}A=0\) for \(n\geq0\), that \(d\) preserves conformal degrees, and that \(V,M\) have compatible cohomological gradings with \(Q\in V^1\), \(A\subset V^0\), and
\[
V^i_{(n)}V^j\subset V^{i+j},\qquad
V^i_{(n)}M^j\subset M^{i+j}.
\]
Then \(H^\bullet(V,d)\) is a vertex \(A\)-algebra via \(A\to H^0(V,d)_0\), and \(H^\bullet(M,d)\) is a positively graded module over it. Since \(d\) commutes with every \(A\)-mode, it extends to the localizations with the same properties.
\begin{lemma}\label{lem:compatibility-with-BRST}
There is a natural isomorphism \(H^\bullet(M,d)_S\cong H^\bullet(M_S,d)\).
\end{lemma}
\begin{proof}
The differential is a homomorphism of \(A\)-loop modules. Exactness of localization 
and Lemma
\ref{lemm:elements-of-localization}  give
\begin{equation}\label{eq:cocycle-sheaf}
Z^i(M)_S\cong Z^i(M_S),\qquad B^i(M)_S\cong B^i(M_S).
\end{equation}
The defining sequence therefore localizes to
\begin{equation}\label{eq:compatibility-BRST}
0\longrightarrow B^i(M)_S\longrightarrow Z^i(M)_S
\longrightarrow H^i(M,d)_S\longrightarrow0.
\end{equation}
Taking the quotient proves the assertion.
\end{proof}

\subsection{\texorpdfstring{The CDO \(\mc D_{\wedge F}^{\mathrm{ch}}\)}{The CDO of the trivial bundle}}\label{subsection:cdo}
For finite-dimensional complex vector spaces \(W_0,W_1\) and the trivial bundle \(F=W_0\times W_1\to W_0\), the associated CDO \cite{GMS03} is the localization sheaf
\[
\mc D_{\wedge F}^{\mathrm{ch}}
:=\bigLoc{W_0}{\mc S(W_0)\otimes\mc E(W_1)},
\]
formed using the vertex \(\mc O(W_0)\)-algebra structure of the \(\beta\gamma bc\) system.

\subsection{Sheaves of BRST reductions on affine GIT quotients}
\label{subsection:Sheaves of BRST reductions}
Let \(X=\on{Spec}A\) be a variety with an action of a complex reductive group \(G\), and set \(\g=\on{Lie}(G)\). For an invariant symmetric form \(\kappa\) on \(\g\), write \(\widehat\g_\kappa=\g((t))\oplus\C\mathbf1\), with bracket
\[
[xf,yg]=[x,y]fg+\on{Res}_{t=0}(g\,df)\kappa(x,y)\mathbf1,
\qquad[\mathbf1,\widehat\g_\kappa]=0.
\]
Its universal affine vertex algebra is
\[
V^\kappa(\g)=U(\widehat\g_\kappa)
\underset{U(\g[[t]]\oplus\C\mathbf1)}\otimes\C,
\]
where \(\g[[t]]\) acts trivially on \(\C\) and \(\mathbf1\) acts as the identity. We identify \(x\in\g\) with \((xt^{-1})\otimes1\).

Suppose \(V\) is a vertex \(A\)-algebra with \(A\subset V_0\) and a graded homomorphism
\begin{equation}\label{chiral-moment-map}
V^{-\kappa_\g}(\g)\longrightarrow V,\qquad\xi\longmapsto J^\xi\in V_1,
\end{equation}
where \(\kappa_\g\) is the Killing form, \(J^\xi_{(0)}A\subset A\), and this action agrees with the infinitesimal \(G\)-action. Every positively graded \(V\)-module is then a smooth \(\affg_{-\kappa_\g}\)-module. Its semi-infinite cohomology \cite{Fei} is computed by
\[
C^{\infty/2+\bullet}(\affg_{-\kappa_\g},M)
=M\otimes\mc E(\g\oplus\g^*),\qquad d=Q_{(0)},
\]
where, for dual bases \(\{\xi_i\}\subset\g\) and \(\{\xi_i^*\}\subset\g^*\),
\[
Q=\sum_i:J^{\xi_i}c^{\xi_i^*}:
-\frac12\sum_{i,j}:c^{\xi_i^*}c^{\xi_j^*}b^{[\xi_i,\xi_j]}:.
\]
The cohomological grading is
\begin{equation}\label{eq:cohomological-grading}
\deg M=0,\qquad\deg c^{\xi_i^*}=1,\qquad\deg b^{\xi_i}=-1.
\end{equation}
The cohomology of \(M\) is naturally a module over that of \(V\), and inherits the conformal grading. There is a level-zero homomorphism
\begin{equation}\label{eq:level-zero-action}
V^0(\g)\longrightarrow C^{\infty/2+\bullet}(\affg_{-\kappa_\g},V),
\quad\xi_i\longmapsto\widehat J^{\xi_i}
:=d b^{\xi_i}=J^{\xi_i}-\sum_j:c^{\xi_j^*}b^{[\xi_i,\xi_j]}:,
\end{equation}
whose induced action on cohomology is trivial. The \emph{relative complex} is
\begin{equation}\label{eq:relative-complex}
C^{\infty/2+\bullet}(\affg_{-\kappa_\g},\g,M)
:=\{u\in C^{\infty/2+\bullet}(\affg_{-\kappa_\g},M)
\mid\widehat J^\xi_{(0)}u=b^\xi_{(0)}u=0\ \ (\xi\in\g)\}.
\end{equation}
Its cohomology, denoted by \(H^{\infty/2+\bullet}(\affg_{-\kappa_\g},\g,M)\), is the relative semi-infinite cohomology \cite{FGZ}.

The natural vertex algebra homomorphism
\begin{equation}\label{map:invariants-to-relative-cohomology}
V^{\g[t]}\longrightarrow H^{\infty/2+0}(\affg_{-\kappa_\g},\g,V),
\qquad a\longmapsto[a\otimes\mathbf1],
\end{equation}
restricts to a ring homomorphism on \(A^G\subset V^{\g[t]}\). Thus relative cohomology is a vertex \(A^G\)-algebra, and the relative cohomology of \(M\) is a loop module over \(A^G\).

Let \(\pi\colon X\to X\sslash G=\on{Spec}(A^G)\) be the quotient map. Denote by \(\pi_*\mc H^{\infty/2+\bullet}(\affg_{-\kappa_\g},\g,\Loc{X}{M})\) the sheaf associated with the presheaf
\[
\mathcal U\longmapsto H^{\infty/2+\bullet}
\bigl(\affg_{-\kappa_\g},\g,\Loc{X}{M}(\pi^{-1}(\mathcal U))\bigr).
\]
\begin{proposition}\label{Pro:pushforward-is-a-sheaf}
There is a natural isomorphism
\[
\pi_*\mc H^{\infty/2+\bullet}(\affg_{-\kappa_\g},\g,\Loc{X}{M})
\cong\bigLoc{X\sslash G}{H^{\infty/2+\bullet}(\affg_{-\kappa_\g},\g,M)}.
\]
\end{proposition}
\begin{proof}
On a principal open \((X\sslash G)_f\), with \(f\in A^G\), all operators defining the relative complex commute with the \(A^G\)-modes. Exactness of localization therefore identifies the localized relative complex with the relative complex of \(M_f\). Lemma~\ref{lem:compatibility-with-BRST} identifies their cohomology, and these identifications commute with restriction.
\end{proof}
\begin{remark}\label{rem:BRSTvsRelativeBRST}
Let \(M\) be a smooth \(\widehat\g_{-\kappa_\g}\)-module whose \(\g\)-action is locally finite and semisimple. There is a natural isomorphism of cohomologically graded vector spaces
\[
H^{\infty/2+\bullet}(\affg_{-\kappa_\g},M)
\cong H^{\infty/2+\bullet}(\affg_{-\kappa_\g},\g,M)
\otimes H_{\mathrm{Lie}}^\bullet(\g,\C),
\]
where \(H_{\mathrm{Lie}}^\bullet(\g,\C)\) denotes Lie algebra cohomology with trivial coefficients. The proof of \cite[Proposition 3.2]{A18} (see also \cite{AMbook}) uses the Hochschild--Serre spectral sequence and applies also to reductive \(\g\) under these hypotheses. The isomorphism is compatible with automorphisms commuting with the affine action.
\end{remark}

\section{The Faithfulness Theorem}\label{sec:faithfulness}
We prove the main lemma and the general Faithfulness Theorem~\ref{theorem:iso}; its geometric form appears in Theorem~\ref{th:faithfulness-2}.

We use quasi-conformal structures in the sense of an action of the vector fields regular on the formal disc; see \cite[Section 6.2.4]{Fre}. We allow rational conformal weights. Explicitly, such a structure consists of even operators \(L_n\), \(n\geq-1\), with \(L_{-1}=T\), \(L_n\vac=0\), and
\begin{equation}\label{eq:quasi-conformal-covariance}
[L_m,L_n]=(m-n)L_{m+n},\qquad
[L_n,a_{(k)}]=\sum_{j=-1}^{n}\binom{n+1}{j+1}
(L_ja)_{(n+k-j)},
\end{equation}
where the first identity is understood for \(m+n\geq-1\).
The operator \(L_0\) is semisimple on \(V\), and the positive operators act locally nilpotently. In the nonnegatively graded setting below, the latter condition follows from the grading.
We write \(L\) for the structure, without requiring a state \(L\in V\).
A compatible module is supplied with operators \(L_n^M\), \(n\geq-1\), satisfying the same vector-field relations and covariance identity, with the right-hand side acting through the module fields.
Morphisms intertwine these operators as well as the vertex algebra action.
For a structure induced by a conformal vector, we use its usual modes on modules. A vector \(a\) is primary of weight \(\Delta\) if \(L_0a=\Delta a\) and \(L_na=0\) for \(n>0\).

\subsection{An annihilation lemma}
\begin{lemma}\label{lem:vanish}
Let \(V\) have a quasi-conformal structure \(L\), and let \(M\) be a compatible module on which \(L_0\) acts locally finitely with lower-bounded spectrum.
Let \(F,G\in V\) be parity-homogeneous primary vectors
of real weights \(\Delta_F,\Delta_G\), respectively.
Assume that
\[
\Delta_F+\Delta_G\in\mathbb Z,\qquad
\Delta_F,\Delta_G\ne\tfrac12,\qquad
F_{(\Delta_F+\Delta_G-1)}G=C_0\mathbf1
\]
for some \(C_0\in\mathbb C^\times\).
If \(m\in M\) satisfies
\[
L_nm=0\ (n\geq1),\qquad
F_{(s)}m=0\ (s\geq\Delta_F),\qquad
G_{(l)}m=0\ (l\geq\Delta_G),
\]
and \(F_{(-1)}m=G_{(-1)}m=0\), then \(m=0\).
\end{lemma}

\begin{proof}
For every integer \(l\) with \(0\leq l<\Delta_G\),
the primary-field commutator formula gives
\[
0=[L_{l+1},G_{(-1)}]m
=\bigl((l+2)\Delta_G-(l+1)\bigr)G_{(l)}m.
\]
The coefficient can vanish in this range only if
\(l=0\) and \(\Delta_G=\tfrac12\).
Thus \(G_{(l)}m=0\) for all \(l\geq0\).
Similarly, \(F_{(k)}m=0\) for all \(k\geq0\).

Set \(n=\Delta_F+\Delta_G-1\).
By \eqref{identity:Bor},
\[
\begin{aligned}
C_0m
&=(F_{(n)}G)_{(-1)}m\\
&=\sum_{j\geq0}(-1)^j\binom nj
\left(
F_{(n-j)}G_{(-1+j)}m
-(-1)^{p(F)p(G)+n}G_{(n-1-j)}F_{(j)}m
\right)=0.
\end{aligned}
\]
Since \(C_0\ne0\), we obtain \(m=0\).
\end{proof}

\subsection{Vertex algebras with two quasi-conformal structures}
Let \(V\) have two quasi-conformal structures \(L,L'\), each defining a
\(\mathbb Q_{\geq0}\)-grading, with commuting grading operators \(L_0,L'_0\).
For a \(V\)-module \(M\), write \(M_{L,d}\) and \(M_{L',d}\)
for the generalized eigenspaces of \(L_0\) and \(L'_0\), respectively.
Denote by \(V\on{-mod}_{L,L',\geq0}\) the category of compatible modules for both structures, with commuting, locally finite grading operators and lower-bounded spectra. Morphisms intertwine both quasi-conformal actions.
For even elements \(f\in V_{L,0}\) and \(g\in V_{L',0}\),
the localizations \(M_f\) and \(M_g\) are therefore defined.

\begin{theorem}\label{theorem:iso}
Let \(V,L,L'\) be as above.
Suppose that there exist even elements
\(f\in V_{L,0}\) and \(g\in V_{L',0}\) such that \(f\) is a
primary vector with respect to \(L'\) of weight
\(\lambda=p/q>0\), where \(p,q\) are coprime positive integers.
Assume that, for every sufficiently large \(N\),
\begin{equation}\label{eq:non-zero-coefficient}
(f^{Nq})_{(Np-1)}g^{Nq}=C_N\mathbf1,
\qquad C_N\in\mathbb C^\times.
\end{equation}
Then:
\begin{enumerate}[label=\textup{(\arabic*)}]
\item
If \(M\in V\on{-mod}_{L,L',\geq0}\) satisfies
\(M_f=M_g=0\), then \(M=0\).

\item
A morphism \(\varphi\colon M\to M'\) in
\(V\on{-mod}_{L,L',\geq0}\) is injective
\textup{(resp.} surjective\textup{)} if both
\(\varphi_f\) and \(\varphi_g\) are injective
\textup{(resp.} surjective\textup{)}.

\item
If \(M_f\) and \(M_g\) are Noetherian over \(V_f\) and \(V_g\), respectively, every ascending chain of \(V\)-submodules of \(M\) stable under both quasi-conformal actions stabilizes. If both structures are induced by conformal vectors in \(V\), then \(M\) is Noetherian as an ordinary \(V\)-module.

\item
If \(V_f\) and \(V_g\) are simple, then \(V\) is simple.
\end{enumerate}
\end{theorem}

\begin{proof}
For sufficiently large \(N\), set
\(F=f^{Nq}\) and \(G=g^{Nq}\).
Since
\[
[L'_n,f_{(-1)}]=((n+1)\lambda-n)f_{(n-1)}
\]
and \(f_{(j)}V_{L,0}=0\) for \(j\geq0\), we have
\(L'_nF=0\) for \(n>0\), and \(F\) has \(L'_0\)-weight \(Np\).
Also, \(G\in V_{L',0}\), so \(L'_nG=0\) for \(n>0\).
Thus \(F,G\) satisfy the 
primary-vector hypotheses of
of Lemma~\ref{lem:vanish}.

(1) Suppose \(M\ne0\).
Since the weights are bounded below,
 we can choose a nonzero vector
\(m\in M_{L',d}\) such that
\(M_{L',d-j}=0\) for every integer \(j\geq1\).
Choose a lowest weight in a
nonzero coset modulo \(\mathbb Z\).
Since \(F\) and \(G\) have integral \(L'_0\)-weights,
this choice gives
\[
L'_nm=0\ (n\geq1),\qquad
F_{(s)}m=0\ (s\geq Np),\qquad
G_{(l)}m=0\ (l\geq0).
\]
As \(M_f=M_g=0\), Lemma~\ref{lem:in-localization} gives
\(F_{(-1)}m=G_{(-1)}m=0\) for sufficiently large \(N\).
Lemma~\ref{lem:vanish} then implies \(m=0\), a contradiction. Once the two localization kernels are defined, 
the argument requires
only the compatible \(L'\)-action.

(2) Apply (1) to \(\ker\varphi\) or \(\coker\varphi\),
using the exactness of localization.

(3) Let \(M_0\subset M_1\subset\cdots\) be an ascending chain
of \(V\)-submodules of \(M\) stable under both quasi-conformal actions.
The localized chains stabilize because \(M_f\) and \(M_g\)
are Noetherian.
Part~(2), applied to \(M_i\hookrightarrow M_{i+1}\), then
shows that the original chain stabilizes. When both actions come from conformal vectors in \(V\), every \(V\)-submodule is stable under their modes.

(4) Since \(F\) and \(G\) have \(L'_0\)-weights \(Np\) and
zero, respectively, \(G_{(n)}F=0\) for \(n\geq Np\).
Skew-symmetry and \eqref{eq:non-zero-coefficient} give
\begin{equation}\label{eq:by-skew-symmetry}
(g^{Nq})_{(Np-1)}f^{Nq}=(-1)^{Np}C_N\mathbf1.
\end{equation}
Let \(I\subsetneq V\) be an ideal.
If \(I_f\ne0\), simplicity gives \(I_f=V_f\), so
\((V/I)_f=0\).
Lemma~\ref{lem:in-localization}, applied to the image of
\(\mathbf1\) in \(V/I\), implies \(f^r\in I\) for some \(r\).
Choosing \(N\) sufficiently large with \(Nq\geq r\), we obtain
\(f^{Nq}\in I\).
Equation~\eqref{eq:by-skew-symmetry} then gives
\(\mathbf1\in I\), a contradiction.
Thus \(I_f=0\).
The same argument, using \eqref{eq:non-zero-coefficient},
gives \(I_g=0\).
Put \(K=\ker(V\to V_f)\cap\ker(V\to V_g)\), so that \(I\subset K\).
This is a \(V\)-submodule stable under the \(L'\)-action.
Indeed, for \(s=f,g\), \([L'_n,s_{(-1)}]\) is a scalar multiple of \(s_{(n-1)}\), and it commutes with \(s_{(-1)}\).
Hence the \(s_{(-1)}\)-torsion subspace is stable under every \(L'_n\).
The argument of (1), applied to \(K\) with its inherited \(L'\)-action, gives \(K=0\).
Thus \(I=0\), proving simplicity.
\end{proof}

\section{Chiral Hamiltonian reduction and global sections}\label{sec:reduction}

\subsection{The BRST sheaf}\label{subsection:our-sheaf}
Let \(G\) be a connected reductive algebraic group with Lie algebra \(\g\), and let \(W_0,W_1\) be finite-dimensional \(G\)-modules. Set \(\widetilde X=W_0\sslash G\), let \(\pi\colon W_0\to\widetilde X\) be the quotient map, and write
\begin{align}\label{eq:mc{A}}
\mathcal A:=\cD_{\wedge F}^{ch}(W_0),\qquad F=W_0\times W_1\to W_0.
\end{align}
We use the notation of Sections~\ref{subsection:betagamma bc-system} and \ref{subsection:Sheaves of BRST reductions}, abbreviating the generating fields to \(\beta^i,\gamma^i,b^j,c^j\). The currents
\begin{align}\label{eq:J-current}
J^\xi
=-\sum_i:\beta^{\xi v_i}\gamma^{v_i^*}:
+\sum_s:b^{\xi w_s}c^{w_s^*}:
\end{align}
define a homomorphism \(V^{\kappa_{W_0,W_1}}(\g)\to\mathcal A\), where \(\kappa_{W_0,W_1}(\xi,\eta)=-\on{tr}_{W_0}(\xi\eta)+\on{tr}_{W_1}(\xi\eta)\). Throughout this section, assume
\begin{equation}\label{eq:level-condition}
\kappa_{W_0,W_1}=-\kappa_\g,\qquad \on{tr}_{W_0}\xi=\on{tr}_{W_1}\xi\quad(\xi\in\g).
\end{equation}
The signs in \eqref{eq:J-current} give
\[
J^\xi_{(0)}\beta^v=\beta^{\xi v},\quad
J^\xi_{(0)}\gamma^\alpha=-\gamma^{\alpha\circ\xi},\quad
J^\xi_{(0)}b^w=b^{\xi w},\quad
J^\xi_{(0)}c^\delta=-c^{\delta\circ\xi}.
\]
The double contractions give the form \(\kappa_{W_0,W_1}\) above.
The trace condition, automatic for semisimple \(G\), ensures BRST closedness of the conformal vectors \(L,L'\) below. The construction of Section~\ref{subsection:Sheaves of BRST reductions} gives a sheaf of relative BRST cohomology on \(\widetilde X\). Its degree-zero part is
\begin{equation}\label{eq:def-of-sheaf-V_X}
\boldsymbol{\mc V}_{\widetilde X}
:=\pi_*\mc H^{\infty/2+0}(\affg_{-\kappa_\g},\g,\mc D_{\wedge F}^{ch}).
\end{equation}
Both the full cohomology sheaf and its degree-zero part are conformal, with conformal vector
\begin{equation}\label{eq:L}
L=\sum_i:\beta^{v_i}\partial\gamma^{v_i^*}:
-\sum_j:b^{w_j}\partial c^{w_j^*}:
-\sum_k:b^{\xi_k}\partial c^{\xi_k^*}:.
\end{equation}
Indeed, \(Q_{(0)}L=\widehat J^\xi_{(0)}L=b^\xi_{(0)}L=0\). By Proposition~\ref{Pro:pushforward-is-a-sheaf}, global sections identify with relative cohomology. In particular,
\begin{equation}\label{eq:def-of-V}
\mathbf V:=\Gamma(\widetilde X,\boldsymbol{\mc V}_{\widetilde X})
\cong H^{\infty/2+0}(\affg_{-\kappa_\g},\g,\mathcal A)
\end{equation}
is a conformal vertex \(\mc O(\widetilde X)\)-algebra.

\subsection{Fourier transform}\label{subsection:Fourier-transform}
Choose a real structure \(\sigma\) on \(G\), compatible real forms \((W_i)_{\mathbb R}\), and a positive-definite Hermitian form on \(W_0\oplus W_1\), conjugate-linear in the first argument, such that \(W_0\perp W_1\), \((\bar x,\bar y)=\overline{(x,y)}\), and
\[
(\xi x,y)=-(x,\omega(\bar\xi)y),\qquad \bar\xi=d\sigma(\xi),
\]
where \(\omega\) is the Chevalley involution. Thus \(\overline{gv}=\sigma(g)\bar v\). Choose orthonormal real bases \(\{v_i\},\{w_j\}\), with dual bases \(v_i^*=(v_i,\cdot)\), \(w_j^*=(w_j,\cdot)\). Extend the form to \(\mc O(T^*W_0)=\on{Sym}(W_0\oplus W_0^*)\), using the induced orthogonal direct-sum form, by
\[
(x_1\cdots x_k,y_1\cdots y_k)
=\sum_{\tau\in S_k}\prod_i(x_i,y_{\tau(i)}),
\]
with distinct polynomial degrees orthogonal and \((\mathbf1,\mathbf1)=1\).

The conjugate-linear Fourier isomorphism
\(\Phi\colon\mathcal A\to\cD_{\wedge F^*}^{ch}(W_0^*)\) is defined by
\begin{equation}\label{eq:Phi-for-cdo}
\gamma^{v_i^*}\mapsto-\beta^{v_i^*},\quad
\beta^{v_i}\mapsto\gamma^{v_i},\quad
b^{w_j}\mapsto c^{w_j},\quad c^{w_j^*}\mapsto b^{w_j^*}.
\end{equation}
Its restriction to polynomial functions satisfies \(\Phi(gf)=\Omega(\sigma(g))\Phi(f)\), where \(d\Omega=\omega\); hence it identifies \(\mc O(W_0)^G\) with \(\mc O(W_0^*)^G\). The basis-dependent complex-linear identification
\begin{equation}\label{eq:identification-dual}
\beta^{v_i^*}\mapsto\beta^{v_i},\quad
\gamma^{v_i}\mapsto\gamma^{v_i^*},\quad
b^{w_j^*}\mapsto b^{w_j},\quad c^{w_j}\mapsto c^{w_j^*}
\end{equation}
identifies the dual free field algebra with \(\mathcal A\). Composing it with \(\Phi\) gives a conjugate-linear automorphism \(\mathring{\mf F}\). Define the complex-linear Fourier transform by \(\mf F(a)=\overline{\mathring{\mf F}(a)}\). Explicitly,
\[
\mf F(\beta^i)=\gamma^i,\quad \mf F(\gamma^i)=-\beta^i,
\quad\mf F(b^j)=c^j,\quad\mf F(c^j)=b^j.
\]
Thus \(\mf F^4=1\); its square is minus the identity on the \(\beta\gamma\) generators and the identity on the \(bc\) generators. Moreover, \(\mf F(J^\xi)=J^{\omega(\xi)}\) and \(\mf F(\mc O(W_0)^G)=\mc O(W_0^*)^G\).

Extend \(\mf F\) to the BRST complex by \(b^\xi\mapsto b^{\omega(\xi)}\) and \(c^\alpha\mapsto c^{\omega^*(\alpha)}\). It preserves the relative subcomplex and satisfies \(\mf F(Q)=Q\), so it induces a complex-linear automorphism of relative cohomology. Consequently,
\[
L':=\mf F(L)=L-TH,\qquad
H=\sum_i:\beta^i\gamma^i:-\sum_j:b^jc^j:
\]
is another conformal vector, and \(\mathbf V\) is nonnegatively graded with respect to both \(L\) and \(L'\). The \emph{canonical conformal vector}
\begin{equation}\label{eq:canonical-conformal-vector}
L_{\mathrm{can}}:=\tfrac12(L+L')=L-\tfrac12TH
\end{equation}
makes \(\mathbf V\) conical and \(\frac12\mathbb Z_{\geq0}\)-graded
with central charge
\begin{equation}\label{eq:cc}
c_{\mathbf V}=-\dim W_0+\dim W_1-2\dim\g.
\end{equation} 

Set \(\widetilde X^*=W_0^*\sslash G\), with quotient map \(\pi^\vee\). Applying the construction to the dual representation gives \(\boldsymbol{\mc V}^*_{\widetilde X^*}\). Under \eqref{eq:identification-dual}, with the corresponding Chevalley twist on the ghost factors, this is the localization sheaf associated with \(\mathbf V\) as a vertex \(\mc O(\widetilde X^*)\)-algebra with conformal vector \(L'\).

\subsection{The commutant and its quasi-conformal structures}\label{subsection:Invariant subalgebras}
Write \(\mathcal C:=\mathcal A^{\g[t]}\).
The ambient free field conformal vectors are
\begin{equation}\label{eq:conformal-vector-of-invariant-subagelrba}
\widetilde L=\sum_i:\beta^{v_i}\partial\gamma^{v_i^*}:
-\sum_j:b^{w_j}\partial c^{w_j^*}:,
\qquad \widetilde L'=\mf F(\widetilde L).
\end{equation}
These are states of \(\mathcal A\); they need not belong to \(\mathcal C\).
The trace condition in \eqref{eq:level-condition} makes each \(J^\xi\) primary of weight one for both ambient structures. Thus
\[
[\widetilde L_n,J^\xi_{(m)}]=-mJ^\xi_{(n+m)},
\qquad n\geq-1,
\]
and similarly for \(\widetilde L'\).
For \(m\geq0\), the right-hand side annihilates \(\mathcal C\): when \(n=-1,m=0\), its coefficient is zero.
Consequently, the operators \(\widetilde L_n,\widetilde L'_n\), \(n\geq-1\), preserve \(\mathcal C\) and define two quasi-conformal structures there.
Their grading operators commute and have nonnegative integral spectra.
The inclusions
\[
\mc O(W_0)^G\subset\mathcal C_{\widetilde L,0},
\qquad
\mc O(W_0^*)^G\subset\mathcal C_{\widetilde L',0}
\]
therefore make \(\mathcal C\) a vertex algebra over both invariant rings.

There is a natural homomorphism
\begin{equation}\label{eq:g[t]-invariant}
\varphi:\mathcal C\longrightarrow\mathbf V,
\qquad a\longmapsto[a\otimes\mathbf1].
\end{equation}
\begin{proposition}\label{prop:vs-beta-gamma-bc}
The map \eqref{eq:g[t]-invariant} intertwines the quasi-conformal structures:
\[
\varphi(\widetilde L_na)=L_n\varphi(a),\qquad
\varphi(\widetilde L'_na)=L'_n\varphi(a)
\quad(a\in\mathcal C,\ n\geq-1).
\]
\end{proposition}
\begin{proof}
On the BRST complex, \(L=\widetilde L+L_{\mathrm{gh}}\), where
\(L_{\mathrm{gh}}=-\sum_i:b^{\xi_i}\partial c^{\xi_i^*}:\).
For \(n\geq-1\), \((L_{\mathrm{gh}})_n\mathbf1=0\), so
\[
L_n(a\otimes\mathbf1)=(\widetilde L_na)\otimes\mathbf1.
\]
The modes \(L_n\) commute with the BRST differential and preserve the relative complex. Passing to cohomology proves the first identity; Fourier transform proves the second.
\end{proof}

We also record the inner coset conformal vectors when \(G\) is semisimple. With dual bases for \(-\kappa_\g/2\), put
\[
L_\g=\tfrac12\sum_i:J^{\xi_i}J^{\xi_i^*}:,
\qquad
L^{\mathcal C}:=\widetilde L-L_\g,
\qquad
L^{\mathcal C}_{\mathrm{can}}:=\tfrac12
\bigl(L^{\mathcal C}+\mf F(L^{\mathcal C})\bigr).
\]
The standard coset calculation gives \(L^{\mathcal C}\in\mathcal C\).
Moreover, \((L_\g)_na=0\) for \(a\in\mathcal C\) and \(n\geq-1\), so the modes of \(L^{\mathcal C}\) induce exactly the quasi-conformal structure defined above.
The BRST identity
\begin{equation}\label{eq:re-expression-of-onformal-vector}
Q_{(0)}\left(\tfrac12\sum_i:b^{\xi_i}J^{\xi_i^*}:\right)
=L_\g-\sum_i:b^{\xi_i}\partial c^{\xi_i^*}:
=L-L^{\mathcal C}
\end{equation}
shows that \(\varphi(L^{\mathcal C})=L\), and Fourier transform gives
\(\varphi(\mf F(L^{\mathcal C}))=L'\).
The quasi-conformal construction and Proposition~\ref{prop:vs-beta-gamma-bc} do not require this Sugawara correction and apply to every connected reductive \(G\) satisfying \eqref{eq:level-condition}.

\subsection{Geometric faithfulness}\label{subsection:faithfullth}
\begin{lemma}\label{lem:ne0}
For a nonzero homogeneous \(f\in\mc O(W_0)\) of degree \(k\), we have \(f_{(k-1)}\Phi(f)=(f,f)\mathbf1\ne0\).
\end{lemma}
\begin{proof}
This is the full-contraction term in the free field OPE; the chosen Hermitian form is positive definite.
\end{proof}
For \(U_i=\widetilde X\setminus V(I)\), define \(U_i^*:=\widetilde X^*\setminus V(\Phi(I))\).
\begin{theorem}\label{th:faithfulness-2}
Assume \eqref{eq:level-condition}. Let \(V=\mathbf V\) or \(V=\mathcal C\), using the quasi-conformal structures \((L,L')\) or \((\widetilde L,\widetilde L')\), respectively, and let \(U_0\subset\widetilde X\) be a nonempty affine open subset. For \(M\in V\on{-mod}_{L,L',\geq0}\), set
\[
\mathcal L(M):=\bigl(\widetilde M_{\widetilde X}(U_0),
\widetilde M_{\widetilde X^*}(U_0^*)\bigr).
\]
Then:
\begin{enumerate}[label=\textup{(\arabic*)}]
\item \(\mathcal L(M)=0\) implies \(M=0\).
\item A morphism is injective (resp. surjective) if both components of its image under \(\mathcal L\) are injective (resp. surjective).
\item If the two components of \(\mathcal L(M)\) are Noetherian over the corresponding localized vertex algebras, every ascending chain of submodules stable under both quasi-conformal actions stabilizes. When the actions are induced by conformal vectors in \(V\), \(M\) is Noetherian as an ordinary module.
\item If both localized vertex algebras are simple, then \(V\) is simple.
\end{enumerate}
\end{theorem}
\begin{proof}
Choose a nonempty principal open \(\widetilde X_f\subset U_0\). Replacing \(f\) by \(f\bar f\), we may assume that it has real coefficients in the chosen bases. Thus \(\mf F(f)=\Phi(f)\) under the fixed identification, and \(\widetilde X^*_{\mf F(f)}\subset U_0^*\). If \(f\) is constant, all assertions are immediate; otherwise, let \(\widetilde f\) be its highest-degree component, of degree \(k>0\). The powers \(\widetilde f^N\) and \(\mf F(\widetilde f)^N\) are \(L'\)-primary of weights \(Nk\) and zero, respectively, and have nonzero full contraction by Lemma~\ref{lem:ne0}.

(1) Write \(L,L'\) also for the selected quasi-conformal structures when \(V=\mathcal C\), and put \(H_0=L'_0-L_0\). The commuting, locally finite operators \(L_0,L'_0\) admit a joint generalized eigenspace decomposition. If \(M\ne0\), choose a joint eigenvector \(m\ne0\) whose \(L'_0\)-weight is lowest within its coset modulo \(\mathbb Z\). The vanishing of the two localizations gives \((f^N)_{(-1)}m=(\mf F(f)^N)_{(-1)}m=0\) for large \(N\). Taking the highest, respectively lowest, \(H_0\)-weight components gives the same identities for \(\widetilde f^N\) and \(\mf F(\widetilde f)^N\). Their integral \(L'_0\)-weights and the choice of \(m\) provide the remaining annihilation conditions in Lemma~\ref{lem:vanish}, which forces \(m=0\).

(2)--(3) These follow from (1) and exactness of localization, as in Theorem~\ref{theorem:iso}.

(4) The principal-open localizations \(V_f,V_{\mf F(f)}\) are nonzero localizations of the assumed simple algebras, hence simple. Let \(I\subsetneq V\) be an ideal. If \(I_f\ne0\), then \((V/I)_f=0\), so \(f^N\in I\) for large \(N\).
The mode \((\mf F(\widetilde f)^N)_{(Nk-1)}\) applied to \(f^N\) is a nonzero scalar multiple of the vacuum: the top-degree component gives the full contraction of Lemma~\ref{lem:ne0}, while every lower-degree component has too few factors to contribute.
This contradicts properness of \(I\). Similarly, \(I_{\mf F(f)}=0\).
The common kernel \(K\) of the two localization maps from \(V\) is stable under both quasi-conformal actions. To see this, for \(s=f,\mf F(f)\), the commutator of any nonnegative quasi-conformal operator with \(s_{(-1)}\) is a mode of a polynomial in the same mutually commuting free fields, and hence commutes with \(s_{(-1)}\). The same assertion holds for the translation operator. Thus the \(s_{(-1)}\)-torsion subspace is stable under these operators.
Part~(1) applied to \(K\) gives \(K=0\); since \(I\subset K\), we have \(I=0\).
\end{proof}
The same argument applies to the full relative BRST cohomology.

\subsection{BRST vanishing and simplicity}
Assume that a proper \(G\)-stable closed subset \(Z\subset W_0\) satisfies:
\begin{enumerate}[label=(A\arabic*)]
\item \(G\) acts freely on \(U:=W_0\setminus Z\);\label{A1}
\item every orbit in \(U\) is closed in \(W_0\).\label{A2}
\end{enumerate}
\begin{theorem}\label{Th:vanishing-and-simplicity}
We have \(H^{\infty/2+i}(\affg_{-\kappa_\g},\g,\mathcal A)=0\) for \(i\ne0\), and \(\mathbf V\) is simple.
\end{theorem}
\begin{proof}
By Theorem~\ref{th:faithfulness-2}, it suffices to prove
vanishing in nonzero degrees and simplicity in degree zero
on a pair of corresponding nonempty affine opens.
Choose \(x\in U\). Luna's slice theorem \cite{BKS} gives
an affine slice \(\mathcal S\) through \(x\) and a Cartesian
diagram with \'etale horizontal arrows:
\begin{equation}\label{eq:luna}
\begin{CD}
G\times\mathcal S @>{\varphi}>> W_0\\
@V{\on{pr}}VV @VV{\pi}V\\
\mathcal S @>{\bar\varphi}>>\widetilde X.
\end{CD}
\end{equation}
Choose a connected component of \(\mathcal S\) dominating its image. After shrinking this image to a nonempty affine open \(U_0\subset U/G\), the slice map is finite \'etale. Taking the normalization in the Galois closure of the corresponding function-field extension and shrinking \(U_0\) further gives a connected finite \'etale Galois cover \(\widetilde{\mathcal S}\to U_0\) factoring through \(\mathcal S\). We may also arrange that \(E|_{U_0}\) is trivial. Write \(H\) for its Galois group.
Then the diagram
\begin{equation}\label{eq:Galois}
\begin{CD}
G\times\widetilde{\mathcal S}@>>>\pi^{-1}(U_0)\\
@VVV @VV{\pi}V\\
\widetilde{\mathcal S}@>>>U_0
\end{CD}
\end{equation}
is Cartesian. By finite \'etale descent for CDOs
\cite[3.7.1]{M}, sections downstairs are identified with
the \(H\)-invariant sections upstairs.

After further shrinking \(U_0\), the local classification of CDOs \cite{GMS,GMS04} gives an identification compatible with the affine action:
\begin{equation}\label{eq:local-cdo}
\mc D_{\wedge F}^{ch}(G\times\widetilde{\mathcal S})
\cong
\cD^{ch}_{G,-\operatorname{tr}_{W_0}}(G)
\otimes\cD^{ch}_{\widetilde{\mathcal S}}
(\widetilde{\mathcal S})\otimes\cE(W_1),
\end{equation}
Here  \(J^\xi\) corresponds to the sum of the group current and the \(bc\) current. The latter has level \(\operatorname{tr}_{W_1}\), so the group current has level
\(\kappa_{W_0,W_1}-\operatorname{tr}_{W_1}=-\operatorname{tr}_{W_0}\).
Since the \(\g[[t]]\)-action on \(\cE(W_1)\) integrates to \(G[[t]]\),
the transfer isomorphism \cite{AMbook} (see also \cite[Section 5]{AG})
\begin{equation}\label{eq:AG-iso}
\cD^{ch}_{G,-\operatorname{tr}_{W_0}}(G)\otimes\cE(W_1)
\cong
\cD^{ch}_{G,-\kappa_\g}(G)\otimes\cE(W_1)
\end{equation}
identifies the diagonal affine action on the left with
the action on the first factor on the right.
Transport the \(H\)-descent action through these identifications. The deck transformations fix the currents pulled back from \(\pi^{-1}(U_0)\), so their action commutes with every \(J^\xi_{(n)}\). Extending it trivially to the ghost factor, it preserves the relative BRST complex and commutes with its differential. The relative form of the calculation in \cite[Section 5]{AG}, using Remark~\ref{rem:BRSTvsRelativeBRST} (see also \cite[Section 5]{A18}), therefore gives an \(H\)-equivariant isomorphism
\begin{equation}\label{eq:local-BRST}
H^{\infty/2+i}(\affg_{-\kappa_\g},\g,
\mc D_{\wedge F}^{ch}(G\times\widetilde{\mathcal S}))
\cong\delta_{i,0}\,
\cD^{ch}_{\widetilde{\mathcal S}}
(\widetilde{\mathcal S})\otimes\cE(W_1),
\end{equation}
Since taking \(H\)-invariants is exact, descent yields
\begin{equation}\label{eq:locally-cdo}
\pi_*\mc H^{\infty/2+i}
(\affg_{-\kappa_\g},\g,\mc D_{\wedge F}^{ch})(U_0)
\cong\delta_{i,0}\,\mc D_{U_0}^{ch}(U_0)\otimes\cE(W_1)
\cong\delta_{i,0}\,\mc D_{\wedge E}^{ch}(U_0).
\end{equation}
Since \(U_0\) is smooth and irreducible, the associated scheme of \(\mc D_{U_0}^{ch}(U_0)\) is the smooth symplectic variety \(T^*U_0\), and
\(\operatorname{gr}_F\mc D_{U_0}^{ch}(U_0)\cong\mc O(J_\infty T^*U_0)\).
Thus \(\mc D_{U_0}^{ch}(U_0)\) is simple by \cite[Corollary 9.3]{AM}. Since \(E|_{U_0}\) is trivial and \(\cE(W_1)\) is simple, their tensor product \(\mc D_{\wedge E}^{ch}(U_0)\) is also simple.
 The dual construction similarly gives
\begin{equation}\label{eq:local-description-dual}
\pi_*^\vee\mc H^{\infty/2+i}
(\affg_{-\kappa_\g},\g,\mc D_{\wedge F^*}^{ch})(U_0^*)
\cong\delta_{i,0}\,\mc D_{\wedge E^*}^{ch}(U_0^*).
\end{equation}
The assertion follows.
\end{proof}

\subsection{Restriction to the smooth quotient}
Define
\begin{align}
X&:=U/G\subset\widetilde X,\label{eq:def-of-X}\\
E&:=(U\times W_1)/G\longrightarrow X.\label{eq:def-of-E}
\end{align}
A CDO associated with \(E\) is locally of the form \(\cD^{ch}_{U_0}\otimes\cE(W_1)\) on a trivializing open \cite{GMS}. The local calculation 
 \eqref{eq:locally-cdo}
 gives the following.
\begin{theorem}\label{th:restriction-is-a-cdo}
The restriction of \(\boldsymbol{\mc V}_{\widetilde X}\) to \(X\) is a CDO \(\mc D_{\wedge E}^{ch}\).
\end{theorem}
Dually, set
\begin{equation}\label{eq:def-of-X*}
U^*:=W_0^*\setminus V(\Phi(I_Z)),\qquad
X^*:=U^*/G,\qquad E^*:=(U^*\times W_1^*)/G,
\end{equation}
where \(I_Z\) is the ideal of \(Z\). The dual sheaf restricts to \(\cdo{\wedge E^*}\) on \(X^*\).

\subsection{The simple quotient of the invariant subalgebra}
\begin{theorem}\label{surj:generator}
Assume \eqref{eq:level-condition} and \ref{A1}--\ref{A2}. The map \(\mathcal C\to\mathbf V\) is surjective. Thus \(\mathbf V\) is a simple quotient of \(\mathcal C\).
\end{theorem}
\begin{proof}
The local transfer isomorphism \eqref{eq:AG-iso} shows that every degree-zero BRST class on \(U_0\) has a \(\g[t]\)-invariant representative. Hence the map is surjective after restriction to \(U_0\), and likewise on \(U_0^*\). Apply Theorem~\ref{th:faithfulness-2}(2), using Proposition~\ref{prop:vs-beta-gamma-bc} to compare the quasi-conformal actions. Both source and target have nonnegative gradings, and the cokernel inherits the compatible actions, so part~(1) applies to this cokernel.
\end{proof}

\subsection{Cotangent reductions and extension of functions}
Let \(\mu\colon T^*W_0=W_0\times W_0^*\to\g^*\) be given by \(\mu(v,w)(\xi)=\langle w,\xi v\rangle\). Its zero fiber has moment-map ideal
\begin{equation}\label{eq:def-of-I_0}
I_0=(\mu^*(\xi)\mid\xi\in\g),\qquad
\mu^*(\xi)=\sum_i(\xi v_i)v_i^*.
\end{equation}
Write \(\cB=\mu^{-1}(0)\cap(U\times W_0^*)\) and \(\cB^*=\mu^{-1}(0)\cap(W_0\times U^*)\).
\begin{lemma}\label{lem:smooth}
The variety \(\cB\) is smooth and irreducible of dimension \(2\dim W_0-\dim\g\).
\end{lemma}
\begin{proof}
The infinitesimal action embeds \(U\times\g\) as a rank-\(\dim\g\) subbundle of \(U\times W_0\), by \ref{A1}. Its annihilator is precisely \(\cB\), a vector bundle over the smooth irreducible variety \(U\).
\end{proof}
\begin{lemma}\label{lemma:smooth}
The singular locus of \(\mu^{-1}(0)\) is contained in \(\mu^{-1}(0)\cap(Z\times Z^*)\).
\end{lemma}
\begin{proof}
Both \(\cB\) and \(\cB^*\) are smooth, by Lemma~\ref{lem:smooth} and its dual.
\end{proof}
\begin{lemma}\label{lem:reduction}
If a reductive group acts freely on a smooth variety \(W\) and admits a smooth geometric quotient \(Y=W/G\) with \(W\to Y\) a principal bundle, then \(T^*Y\cong\mu^{-1}(0)/G\).
\end{lemma}
\begin{proof}
The cotangent sequence identifies \(\mu^{-1}(0)\) with the pullback of \(T^*Y\): a covector lies in the zero fiber exactly when it annihilates the orbit directions. Taking the geometric quotient gives the assertion.
\end{proof}
\begin{lemma}\label{lem:smoothopen}
The cotangent bundles \(T^*E\) and \(T^*X\) are smooth irreducible open subvarieties of \(T^*F\ssslash_0G\) and \(T^*W_0\ssslash_0G\), respectively.
\end{lemma}
\begin{proof}
Apply Lemma~\ref{lem:reduction} to the free loci \(U\times W_1\) and \(U\). In particular, \(T^*X\cong\cB/G\).
\end{proof}
We impose the following additional assumption:
\begin{enumerate}[label=(A\arabic*),start=3]
\item \(\mu^{-1}(0)\cap(U\times Z^*)\) has codimension at least two in \(\cB\).\label{A3}
\end{enumerate}
\begin{lemma}\label{lem:restrictions}
Restriction to the common open locus gives isomorphisms
\begin{enumerate}[label=\textup{(\roman*)}]
\item \(\mc O(T^*X)\cong\mc O(T^*X^*)\);
\item \(\mc O(J_\infty T^*X)\cong\mc O(J_\infty T^*X^*)\);
\item the analogous isomorphisms for \(T^*E,T^*E^*\), their arc spaces, and \(\wedge J_\infty T^*E,\wedge J_\infty T^*E^*\).
\end{enumerate}
\end{lemma}
\begin{proof}
Normality and \ref{A3} give
\begin{equation}\label{eq:O(T*X))}
\mc O(T^*X)=\mc O(\cB)^G
=\mc O(\mu^{-1}(0)\cap(U\times U^*))^G.
\end{equation}
The dual argument proves (i). For (ii), the inverse image of the common open locus under \(J_m(T^*X)\to T^*X\) has complement of the same codimension, since this map is locally an affine-space bundle. Hartogs extension on the smooth jet schemes identifies the rings
of functions for every \(m\); passing to arcs proves (ii). The same argument, applied to the relevant vector bundles and exterior-power sheaves, proves (iii).
\end{proof}
\begin{lemma}\label{lem:codim2}
Both \(Z\subset W_0\) and \(\widetilde X\setminus X\) have codimension at least two. Hence \(\mc O(X)=\mc O(W_0)^G\).
\end{lemma}
\begin{proof}
By \ref{A3}, \(\cB\cap\cB^*\) is nonempty and open in \(\cB^*\). Since \(\cB^*\to U^*\) is a vector bundle, the projection \(\cB\to W_0^*\) is dominant. Suppose that \(Z^*\) contains a prime divisor defined by an irreducible polynomial \(f\). Its pullback to \(\cB\) is nonzero and is not a unit: a unit would restrict to a constant on every vector-space fiber of \(\cB\to U\), necessarily \(f(0)\), contradicting dominance. Its zero locus is therefore a nonempty divisor in \(\cB\), contradicting \ref{A3}. Thus \(\on{codim}_{W_0^*}Z^*\geq2\), and Fourier duality gives the same bound for \(Z\). Hartogs extension and descent give \(\mc O(X)=\mc O(U)^G=\mc O(W_0)^G\). Since \(\widetilde X\) is normal and affine, this equality also excludes a divisorial component of its complement.
\end{proof}

\subsection{Filtrations on global sections}
The sheafwise Li filtration satisfies
\begin{equation}\label{eq:gr D}
\on{gr}_F\cD_{\wedge E}^{ch}\cong \cO_{\wedge J_\infty T^*E}.
\end{equation}
Here and below, we regard \(\cO_{\wedge J_\infty T^*E}\) as a sheaf on \(X\), suppressing pushforward notation. Define \(\mc R_p\Gamma(X,\cD_{\wedge E}^{ch}):=\Gamma(X,F_p\cD_{\wedge E}^{ch})\). Left exactness applied to \(0\to F_{p-1}\to F_p\to F_p/F_{p-1}\to0\) gives a natural embedding of vertex Poisson algebras
\begin{equation}\label{eq:embedding-grR}
\on{gr}_{\mc R}\Gamma(X,\cD_{\wedge E}^{ch})\hookrightarrow\Gamma(X, \cO_{\wedge J_\infty T^*E}).
\end{equation}
There is also a natural vertex Poisson homomorphism \(\on{gr}_F\Gamma(X,\cD_{\wedge E}^{ch})\to\on{gr}_{\mc R}\Gamma(X,\cD_{\wedge E}^{ch})\).

\begin{lemma}\label{lemma:L0-action-on-associated graded}
We have \(L,L'\in\mc R_1\Gamma(X,\cD_{\wedge E}^{ch})_2\).
The operators \(L_0,L'_0\) therefore preserve the sheafwise filtration \(F_\bullet\) and the filtration \(\mc R_\bullet\) on global sections. They induce actions on
\(\on{gr}^p_{\mc R}\Gamma(X,\cdo{\wedge E})\) and
\(\Gamma(X,\on{gr}^p_F\cD_{\wedge E}^{ch})\), and \eqref{eq:embedding-grR} intertwines these actions.
\end{lemma}
\begin{proof}
Choose a sufficiently small connected affine open covering \(\{U_\alpha\}\) of \(X\) trivializing \(E\). The local CDO description gives
\[
R_{\cdo{\wedge E}(U_\alpha)}
\cong\Gamma(U_\alpha,\cO_{\wedge T^*E})
\cong\mc O(T^*U_\alpha)\otimes\bigwedge(W_1\oplus W_1^*).
\]
By nondegeneracy of the canonical Poisson bracket, its Poisson supercenter consists of constants, which have \(L_0\)-weight zero. The images of \(L,L'\) are Poisson central because \(L_{(0)}=L'_{(0)}=T\), and both have \(L_0\)-weight two. Thus their images vanish. The diagonal description of the local \(C_2\)-algebra gives \(L,L'\in F_1\cdo{\wedge E}(U_\alpha)_2\). These local inclusions yield the asserted inclusion in \(\mc R_1\). Finally, \eqref{eq:condition-filtration} implies
\((F_1)_{(1)}F_p\subset F_p\), so both operators preserve the filtrations. Equivariance follows from the naturality of the symbol map.
\end{proof}
\subsection{The global-sections isomorphism}
\begin{theorem}\label{th:no-support}
The restriction map
\begin{equation}\label{eq:restriction-map-of-V}
\mathbf V=\Gamma(\widetilde X,\boldsymbol{\mc V}_{\widetilde X})
\longrightarrow \Gamma(X,\cD_{\wedge E}^{ch})
\end{equation}
is an isomorphism.
\end{theorem}
\begin{prop}\label{prop:non-negatively-graded}
On \(\Gamma(X,\cD_{\wedge E}^{ch})\), the operator \(L_0\) acts semisimply with nonnegative weights, and \(L'_0\) acts locally finitely with nonnegative spectrum.
\end{prop}
\begin{proof}
The \(L_0\)-grading of the CDO sheaf gives a nonnegative grading on its global sections. By Lemma~\ref{lemma:L0-action-on-associated graded}, \eqref{eq:embedding-grR} is \(L'_0\)-equivariant. The common-open-locus identification in Lemma~\ref{lem:restrictions} intertwines the \(L'_0\)-actions on
\(\Gamma(X,\cO_{\wedge J_\infty T^*E})\) and
\(\Gamma(X^*,\cO_{\wedge J_\infty T^*E^*})\). On the latter, \(L'_0\) defines the natural nonnegative grading. Hence \(L'_0\) acts semisimply with nonnegative weights on \(\on{gr}_{\mc R}\Gamma(X,\cD_{\wedge E}^{ch})\).

Since \(L_0\) and \(L'_0\) commute, \(L'_0\) preserves each \(L_0\)-weight space, on which \(\mc R_\bullet\) is a finite filtration.
Local finiteness and nonnegativity of the spectrum therefore lift from its associated graded to the original space. This proves the assertion.
\end{proof}
\begin{proof}[Proof of Theorem~\ref{th:no-support}]
Set $B=\Gamma(X,\cD_{\wedge E}^{ch})$.
Choose a real-coefficient invariant \(f\) with \(\widetilde X_f\subset X\) and \(\widetilde X^*_{\mf F(f)}\subset X^*\) as in the local calculation. Simplicity makes \eqref{eq:restriction-map-of-V} injective. By Proposition~\ref{prop:non-negatively-graded} and faithfulness, it suffices to prove that
\begin{equation}\label{eq:need-to-show}
\mathbf V_f\to B_f,\qquad\mathbf V_{\mf F(f)}\to B_{\mf F(f)}
\end{equation}
are isomorphisms. The first follows from the chain of injections \(\mathbf V_f\hookrightarrow B_f\hookrightarrow\Gamma(\widetilde X_f,\cD_{\wedge E}^{ch})\), whose composite is the local isomorphism \eqref{eq:locally-cdo}.

For the second, use the compatible symbol maps
\begin{equation}\label{eq:series-of-embeddings}
\on{gr}_F\mathbf V\longrightarrow\on{gr}_{\mc R}B
\hookrightarrow\Gamma(X,\cO_{\wedge J_\infty T^*E}).
\end{equation}
After localization at \(f\), they give
\begin{equation}\label{eq:embedding-f}
(\on{gr}_F\mathbf V)_f\longrightarrow(\on{gr}_{\mc R}B)_f
\hookrightarrow\Gamma(X, \cO_{\wedge J_\infty T^*E})_f
\hookrightarrow\Gamma(\widetilde X_f,\cO_{\wedge J_\infty T^*E}).
\end{equation}
The composite is an isomorphism by the local CDO description, so every map is an isomorphism. In particular,
\begin{equation}\label{eq:iso-gr-R-loc}
(\on{gr}_F\mathbf V)_f\cong\Gamma(X,\cO_{\wedge J_\infty T^*E})_f.
\end{equation}
Similarly, localization of \eqref{eq:series-of-embeddings} at \(\mf F(f)\) gives
\begin{equation}\label{eq:embedding-f-Fourier}
(\on{gr}_F\mathbf V)_{\mf F(f)}\longrightarrow
(\on{gr}_{\mc R}B)_{\mf F(f)}\hookrightarrow
\Gamma(X,\cO_{\wedge J_\infty T^*E})_{\mf F(f)}.
\end{equation}
Li's decreasing filtration is intrinsic and hence preserved by \(\mf F\). After degreewise reindexing, \(\mf F\) therefore identifies the increasing associated graded algebras defined using \(L\) and \(L'\). By the naturality of symbol maps under restriction and the identifications of Lemma~\ref{lem:restrictions}, the composite in \eqref{eq:embedding-f-Fourier} is the Fourier transform of \eqref{eq:iso-gr-R-loc}, hence an isomorphism. Thus the filtered map \(\mathbf V_{\mf F(f)}\to B_{\mf F(f)}\) is an isomorphism, proving the theorem.
\end{proof}
\begin{cor}
The vertex algebra \(\Gamma(X,\cD_{\wedge E}^{ch})\) is simple and nonnegatively graded with respect to \(L,L'\).
\end{cor}
There is consequently an exact localization functor \(\on{Loc}_X\colon\mathbf V\on{-mod}_{L,\geq0}\to\cdo{\wedge E}\on{-mod}_{L,\geq0}\), sending \(M\) to \(\widetilde M_X:=\widetilde M_{\widetilde X}|_X\). The lowest-weight component of a sheaf module, when it exists, is naturally a module over the sheaf \(\mc D_{\wedge E}\) of differential operators \cite{ACF}.

\subsection{A fiberwise criterion for normality and affinization}
Assume further:
\begin{enumerate}[label=(A\arabic*),start=4]
\item \(T^*W_0\ssslash_0G\) is the affinization of \(T^*X\), equivalently \(\mc O(T^*X)=\mc O(\mu^{-1}(0))^G\).\label{A4}
\end{enumerate}
\begin{lemma}\label{lem:isoHAcond}
Let \(\on{pr}\colon\mu^{-1}(0)\to W_0\) be projection.
\begin{enumerate}
\item If \(\mu^{-1}(0)\cap(\{v\}\times Z^*)\) has codimension at least one in \(\on{pr}^{-1}(v)\) for every \(v\), then \(\mu^{-1}(0)\) is irreducible and normal.
\item If these codimensions are at least two, then \ref{A3} and \ref{A4} hold.
\end{enumerate}
\end{lemma}
\begin{proof}
Write \(Y=\mu^{-1}(0)\), \(Y_v=(\g v)^\perp\subset W_0^*\), and \(d=\dim W_0\). In (1), \(\cB^*\cap Y_v\) is nonempty and dense in every affine fiber. Hence the smooth irreducible open \(\cB^*\) is dense in \(Y\), so \(Y\) is irreducible of dimension \(2d-\dim\g\). The \(\dim\g\) moment-map equations therefore define a complete intersection. It is Cohen--Macaulay and generically reduced, hence reduced.

To check regularity in codimension one, note that \(\cB^*\cap\on{pr}^{-1}(Z)\) is a proper closed subset of \(\cB^*\). Stratifying \(Z\) by fiber dimension shows that \(\dim\on{pr}^{-1}(Z)\leq\dim Y-1\). The fiberwise codimension-one hypothesis then gives
\[
\dim\bigl(Y\cap(Z\times Z^*)\bigr)\leq\dim Y-2.
\]
Apply Lemma~\ref{lemma:smooth} and Serre's criterion \cite[Chapter II, Section 8]{Har}.

In (2), the same fiber-dimension argument gives \(\on{codim}_Y(Y\setminus\cB^*)\geq2\), and \ref{A3} follows by restriction to \(\cB\). Hartogs extension and Lemma~\ref{lem:restrictions} give
\[
\mc O(Y)^G=\mc O(\cB^*)^G=\mc O(\cB)^G=\mc O(T^*X),
\]
which is \ref{A4}.
\end{proof}

\subsection{The associated variety}
Transport the \(\mc R\)-filtration to \(\mathbf V\) using Theorem~\ref{th:no-support}. Its diagonal symbol algebra
\[
\mf R_{\mathbf V}:=\bigoplus_\Delta
\mathbf V_\Delta/\mc R_{\Delta-1}\mathbf V_\Delta
\]
is Poisson. The symbol maps give \(R_{\mathbf V}\twoheadrightarrow \mf R_{\mathbf V}\hookrightarrow\Gamma(X,\mc{O}_{\wedge T^*E})\). 
Passing to the reduced rings gives a Poisson homomorphism
to \(\mc O(T^*X)\), which is identified with the classical
Hamiltonian reduction by \ref{A4}.
Thus we obtain
\begin{equation}\label{eq:surjc-RV}
(R_{\mathbf V})_{\mathrm{red}}\longrightarrow\mc O(T^*W_0\ssslash_0G).
\end{equation}
\begin{thm}\label{thm:surjection-to-classical-invairiant-rings}
The map \eqref{eq:surjc-RV} is surjective.
\end{thm}
\begin{proof}
Put \(P=\mc O(T^*W_0)\). The degreewise Hermitian orthogonal decomposition \(P=I_0\oplus I_0^\perp\) is stable under a compact form of \(G\), hence under \(G\). Thus
\[
(I_0^\perp)^G\cong P^G/I_0^G=\mc O(T^*W_0\ssslash_0G).
\]
Identify \(P\) with \(\C[\beta^i_{(-1)},\gamma^i_{(-1)}]\mathbf1\subset\mathcal A\). For \(f\in(I_0^\perp)^G\), the zero modes \(J^\xi_{(0)}\) vanish by invariance, while \(J^\xi_{(n)}f=0\) for \(n\geq2\). The remaining condition follows from the adjoint identity
\[
(J^\xi_{(1)}f,h)=-(f,\mu^*(\omega(\bar\xi))h),\qquad h\in P.
\]
Indeed, in the chosen orthonormal bases, if \(\rho(\xi)=(A_{ji})\), the first-mode action is
\[
J^\xi_{(1)}|_P=\sum_{i,j}A_{ji}\partial_{\beta^i}\partial_{\gamma^j}.
\]
Its Hermitian adjoint is multiplication by \(-\mu^*(\omega(\bar\xi))\). The right-hand side vanishes since its second argument belongs to \(I_0\). Hence \(f\in\mathcal C\), and its image in \(\mathbf V\) has the prescribed classical symbol. This proves surjectivity.
\end{proof}
The reduced Hamiltonian reduction has finitely many symplectic leaves \cite[Section 2.3]{Lo}.
\begin{conj}\label{mainconj}
We have \(X_{\mathbf V}\cong T^*W_0\ssslash_0G\). In particular, \(\mathbf V\) is quasi-lisse.
\end{conj}
\subsection{A generation criterion}\label{subsection:generators}
\begin{proposition}\label{prop:generation-criterion}
Let \(V\subset\mathbf V\) be a vertex subalgebra stable under the quasi-conformal operators \(L_n,L'_n\), \(n\geq-1\).
Suppose that \(V\) contains both \(\mc O(W_0)^G\) and \(\mc O(W_0^*)^G\), and that the natural symbol map
\[
R_V\longrightarrow R_{\mathbf V}
\longrightarrow\Gamma(X,\mc O_{\wedge T^*E})
\]
is surjective. Then \(V=\mathbf V\).
\end{proposition}
\begin{proof}
Choose a nonzero real-coefficient invariant \(f\in\mc O(W_0)^G\) with \(\widetilde X_f\subset X\), as in the proof of Theorem~\ref{th:faithfulness-2}, and put \(g=\mf F(f)=\Phi(f)\).
Localizing the symbol map at \(f\) gives a surjection
\[
R_{V_f}\cong(R_V)_f
\twoheadrightarrow
\Gamma(\widetilde X_f,\mc O_{\wedge T^*E})
\cong R_{\mathbf V_f}.
\]
Here we use Lemma~\ref{eq:filtration-commutes-with-localization2}, localization of sections of the quasi-coherent symbol sheaf, and the local CDO description of \(\mathbf V_f\).
Since \(\operatorname{gr}_F\mathbf V_f\) is generated as a differential algebra by \(R_{\mathbf V_f}\), the induced map
\[
\operatorname{gr}_F V_f\longrightarrow
\operatorname{gr}_F\mathbf V_f
\]
is surjective; see \eqref{eq:Li-surjection}.
Both localized algebras are nonnegatively \(L_0\)-graded, and the Li filtration is finite in each weight. Surjectivity 
\(V_f\to\mathbf V_f\)
therefore follows by filtered reconstruction.

On the common open locus used in Lemma~\ref{lem:restrictions}, the symbol maps for the two polarizations agree under the classical cotangent identification.
The same lemma identifies their spaces of global symbols, so the assumed surjectivity also holds for the dual symbol map on \(X^*\).
Repeating the preceding argument with \(L'\) and \(g\) gives a surjection
\(V_g\to\mathbf V_g\).

The proof of Theorem~\ref{th:faithfulness-2} applies to \(V\), since it contains both invariant rings and inherits the quasi-conformal actions and full-contraction identities.
Applying part~(2) to the inclusion \(V\hookrightarrow\mathbf V\), which is surjective after localization at both \(f\) and \(g\), gives \(V=\mathbf V\).
\end{proof}

\subsection{Modules over CDOs}
Let \(\cdo{\wedge E}\on{-mod}_{L,\geq0}\) denote the category of modules on which \(L_0\) acts locally finitely with 
spectrum bounded below.

\begin{theorem}[\cite{ACF}]\label{th:ACF}
The weight-zero functor gives an equivalence
\[
\begin{aligned}
\cdo{\wedge E}\on{-mod}_{L,\geq0}
&\simeq \mc D_{\wedge E}\on{-mod},
&\mc M&\longmapsto\mc M_{top}.
\end{aligned}
\]
\end{theorem}
\begin{proof}
Work locally on a sufficiently small affine open
\(\mc U\subset X\) trivializing \(E\), and put
\(A=\cdo{\wedge E}(\mc U)\). Then
\[
\on{Zhu}(A)\cong
\mc D(\mc U)\otimes\on{Cl}(W_1).
\]
This algebra acts faithfully on
\(A_0=\mc O_{\wedge E}(\mc U)\).
Since \(L_0A_0=0\), we have \([L]=0\) in \(\on{Zhu}(A)\).

Let \(M\) be a nonzero lower-truncated \(A\)-module
on which \(L_0\) acts locally finitely.
The lowest generalized weight space in any nonzero coset modulo \(\Z\) is a Zhu module. Since \([L]=0\), this space has weight zero and \(L_0\) acts on it as zero. Hence the spectrum lies in \(\Z_{\geq0}\), and the generalized weight-zero space is the ordinary eigenspace \(M_0\neq0\).

For a \(\on{Zhu}(A)\)-module \(N\), let \(S(N)\) be its
universal induced nonnegatively graded \(A\)-module,
with \(S(N)_0=N\).
For any such \(A\)-module \(M\), the canonical map
\(S(M_0)\to M\) is an isomorphism in degree zero.
Its kernel and cokernel therefore have zero
degree-zero components and hence vanish.
Thus \(S\) and the weight-zero functor are quasi-inverse.
These local equivalences are compatible with restriction
and therefore glue.
\end{proof}
\begin{remark}
Let \(N\) be the exterior-degree Euler operator and set
\[
e:=\prod_{k=1}^{\on{rank}E}\left(1-\frac Nk\right).
\]
Then \(e\) is the projection onto exterior degree zero.
Locally, the Clifford-algebra description shows that
\(e\) is a full idempotent and
\[
e(\mc D_{\wedge E})e\cong\mc D_X.
\]
Thus \(\mc M\mapsto e\mc M\) gives a Morita equivalence
between \(\mc D_{\wedge E}\)-modules and
\(\mc D_X\)-modules.
Under this equivalence, \(\mc O_{\wedge E}\)
corresponds to \(\mc O_X\).
\end{remark}
\subsection{Equivalence of categories}\label{subsection:equivalence}
By Theorems~\ref{th:faithfulness-2} and \ref{th:ACF}, there is a faithful exact functor
\begin{align}
\mathbf V\on{-mod}_{L,L',\geq0}\longrightarrow
\mc D_{\wedge E}\on{-mod}^{\oplus2},\qquad
M\longmapsto\bigl((\widetilde M_X)_{\mathrm{top}},
(\widetilde{\mf F^*M}_X)_{\mathrm{top}}\bigr).
\end{align}
Consider \(\mathbf V\)-modules on which the commuting operators \(L_0,L'_0\) act locally finitely. Such a module \(M\) is called \emph{finitely strongly generated} if it has finitely many simultaneous generalized eigenvectors \(m_1,\ldots,m_t\) such that
\[
M=\operatorname{span}_\C\left\{
 a^1_{(-n_1)}\cdots a^r_{(-n_r)}m_j
 \ \middle|\ r\geq0,\ n_i\geq1,\ 1\leq j\leq t
\right\},
\]
where the \(a^i\in\mathbf V\) are simultaneous homogeneous vectors for \(L_0,L'_0\).
We write \(\mathbf V\on{-mod}^{\mathrm{fin}}\) for the category of these modules.
Finite strong generation implies lower-boundedness for both conformal gradings. Indeed, for \(a\) of \(L_0\)-weight \(\Delta_L(a)\),
\[
[L_0,a_{(-n)}]=(\Delta_L(a)+n-1)a_{(-n)},
\qquad n\geq1,
\]
and the weight increment is nonnegative; the same holds for \(L'_0\).
Thus both spectra are bounded below, and
\(\mathbf V\on{-mod}^{\mathrm{fin}}\) is a full subcategory of
\(\mathbf V\on{-mod}_{L,L',\geq0}\).
Fourier twist preserves this category, since it exchanges \(L,L'\) and preserves mode indices.

We will also use
\[
C_2(M):=\operatorname{span}_\C\{a_{(-2)}m\mid a\in\mathbf V,\ m\in M\}.
\]
For a finitely strongly generated module, \(M/C_2(M)\) is generated over \(R_{\mathbf V}\) by the images of the \(m_j\).

\begin{lemma}\label{lem:finite-module-localization}
For \(M\in\mathbf V\on{-mod}^{\mathrm{fin}}\), the generalized eigenspaces of \(L_{\mathrm{can},0}\) are finite-dimensional, and
\((\widetilde M_X)_{\mathrm{top}}\) is \(\mc O_X\)-coherent.
\end{lemma}
\begin{proof}
The commuting, locally finite operators \(L_0,L'_0\) give a lower-bounded canonical grading on \(M\).
The canonical weight spaces of \(\mathbf V\) are finite-dimensional, as follows from the relative free field BRST construction.
Choose finitely many simultaneous generalized eigenvectors strongly generating \(M\), as in the definition above.
Every nontrivial factor \(a_{(-n)}\), \(n\geq1\), raises canonical weight by
\(\Delta_a+n-1\geq\tfrac12\), since \(\mathbf V_{\mathrm{can},0}=\C\vac\).
There are therefore only finitely many choices in any fixed weight, proving the first assertion.

Cover \(X\) by principal opens \(\widetilde X_f\subset X\).
By Theorem~\ref{th:no-support}, the localized CDO and module are \(\mathbf V_f\) and \(M_f\).
Modulo \(C_2\), loop localization is ordinary localization: inverting \(f_{(-1)}\) induces multiplication by \(\bar f^{-1}\). In particular, there is a natural surjection
\[
(M/C_2(M))_{\bar f}\twoheadrightarrow M_f/C_2(M_f).
\]
Thus the target is finitely generated over \(R_{\mathbf V_f}\).
By Theorem~\ref{th:ACF}, \(M_f\) has nonnegative integral \(L_0\)-weights.
Consequently \(C_2(M_f)_0=0\), and its weight-zero space is finitely generated over
\[
(R_{\mathbf V_f})_0=(\mathbf V_f)_0
=\mc O_{\wedge E}(\widetilde X_f).
\]
This algebra is  locally free of finite rank over \(\mc O_X\). The resulting local finite generation proves coherence.
\end{proof}

We use the following elementary fact about connections whose Euler operator acts locally finitely.
\begin{lemma}\label{lem:euler-connection}
Let \(R=\C[x_1,\ldots,x_d]\), and let \(P\) be a finite \(R\)-module with an integrable connection. If
\[
\mathsf E:=\sum_{i=1}^d x_i\nabla_{\partial_i}
\]
acts locally finitely on \(P\), multiplication induces an isomorphism of connections
\[
R\otimes_\C P^\nabla\xrightarrow{\ \sim\ }P,
\qquad
P^\nabla:=\bigcap_i\ker\nabla_{\partial_i}.
\]
In particular, \(P^\nabla\) is finite-dimensional.
\end{lemma}
\begin{proof}
Choose a finite-dimensional \(\mathsf E\)-stable subspace containing a set of \(R\)-generators of \(P\). Since
\([\mathsf E,x_i]=x_i\), there is a finite set \(\Lambda\subset\C\) such that
\[
\operatorname{Spec}(\mathsf E;P)\subset\Lambda+\Z_{\geq0}.
\]
The identities \([\mathsf E,\nabla_{\partial_i}]=-\nabla_{\partial_i}\) therefore imply
\(\nabla_\partial^\alpha v=0\) for every \(v\in P\) and all sufficiently large \(|\alpha|\).
Thus the Taylor operator
\[
\tau(v):=\sum_{\alpha\in\Z_{\geq0}^d}
\frac{(-x)^\alpha}{\alpha!}\nabla_\partial^\alpha v
\]
is a finite sum and takes values in \(P^\nabla\).
The identities
\[
v=\sum_\alpha\frac{x^\alpha}{\alpha!}
\tau(\nabla_\partial^\alpha v),
\qquad
\tau(x^\alpha h)=
\begin{cases}h,&\alpha=0,\\0,&\alpha\ne0,\end{cases}
\quad(h\in P^\nabla)
\]
give mutually inverse maps between \(P\) and \(R\otimes_\C P^\nabla\).
Finally, \(P^\nabla\cong P/(x_1,\ldots,x_d)P\) is finite-dimensional.
\end{proof}

\begin{lemma}\label{lem:trivial-top-connection}
For \(M\in\mathbf V\on{-mod}^{\mathrm{fin}}\), there is a finite-dimensional vector space \(S\) such that
\[
(\widetilde M_X)_{\mathrm{top}}
\cong\mc O_{\wedge E}\otimes_\C S
\]
as \(\mc D_{\wedge E}\)-modules.
\end{lemma}
\begin{proof}
Put \(T=(\widetilde M_X)_{\mathrm{top}}\) and \(N=eT\), using the Morita idempotent introduced after Theorem~\ref{th:ACF}.
Lemma~\ref{lem:finite-module-localization} makes \(N\) a coherent \(\mc O_X\)-module with an integrable connection.
Write \(p:U\to X\), \(j:U\hookrightarrow W_0\), and \(R=\mc O(W_0)=\C[x_1,\ldots,x_d]\).

We first use the global chiral localization
\(\widetilde M_X=\widetilde M_{\widetilde X}|_X\) to obtain generating sections.
The images of the generalized weight-zero space \(M_{L,0}\) generate \(T\) over \(\mc O_{\wedge E}\), since on every principal affine chart taking weight zero in loop localization gives ordinary localization of these images.
The idempotent \(e\) commutes with \(\mc O_X\) and kills left multiplication by functions of positive exterior degree. Hence the sections \(em\), \(m\in M_{L,0}\), generate \(N\) over \(\mc O_X\).
Moreover, \(H_{(0)}=L'_0-L_0\) is locally finite on \(M_{L,0}\) and commutes with \(e\).
Under the Zhu identification its action on \(N\) is the Euler differential operator of the base; after pullback this becomes
\[
\mathsf E=\sum_i x_i\nabla_{\partial_i}
\]
on \(p^*N\).
Thus \(p^*N\) is generated by global sections on which \(\mathsf E\) acts locally finitely.
By coherence and quasi-compactness of \(U\), finitely many suffice. Their \(\mathsf E\)-orbits span a finite-dimensional stable space
\(K\subset\Gamma(U,p^*N)\).

Set
\[
P_0=RK\subset\Gamma(U,p^*N),\qquad
P=P_0^{**},\qquad (-)^*=\operatorname{Hom}_R(-,R).
\]
The finite \(R\)-module \(P_0\) restricts to \(p^*N\).
Since \(\operatorname{codim}_{W_0}(W_0\setminus U)\geq2\) by Lemma~\ref{lem:codim2}, reflexive extension identifies its double dual with
\[
\widetilde P\cong j_*p^*N.
\]
The connection extends uniquely to \(\widetilde P\) by pushing forward its operators along the constant vector fields; here
\(j_*(p^*N\otimes\Omega^1_U)=\widetilde P\otimes\Omega^1_{W_0}\).
The Leibniz and integrability identities extend from \(U\), so \(P\) is a finite \(R\)-module with an integrable connection.

Its Euler operator is locally finite.
Indeed, \(P_0=RK\) is \(\mathsf E\)-locally finite; the semisimple part of \(\mathsf E\) gives a grading compatible with the polynomial grading of \(R\), and its nilpotent part is an \(R\)-linear nilpotent endomorphism.
Duality preserves these properties: an \(R\)-linear map from \(P_0\) to \(R\) has finitely many homogeneous components, as is seen from its values on finitely many homogeneous generators.
Applying this twice proves local finiteness on \(P=P_0^{**}\).
The resulting operator is the Euler operator of the extended connection, since they agree on \(U\).
Lemma~\ref{lem:euler-connection} therefore gives the natural horizontal-section isomorphism
\[
\widetilde P\cong\mc O_{W_0}\otimes_\C S,
\qquad S=P^\nabla,
\]
with \(S\) finite-dimensional.

Finally, the canonical strong \(G\)-equivariance of \(p^*N\) extends uniquely to \(j_*p^*N\) and preserves \(S\).
Its infinitesimal action on horizontal sections is zero, and connectedness of \(G\) makes its algebraic action on \(S\) trivial.
Restricting to \(U\) and descending gives \(N\cong\mc O_X\otimes S\).
Morita equivalence yields \(T\cong\mc O_{\wedge E}\otimes S\), as required.
\end{proof}

\begin{theorem}\label{th:equivalence}
Assume \eqref{eq:level-condition} and \ref{A1}--\ref{A3}. Then the functor
\[
\mathbf{Vect}_{\mathrm{fd}}\xrightarrow{\ \sim\ }
\mathbf V\on{-mod}^{\mathrm{fin}},\qquad
S\longmapsto\mathbf V\otimes S
\]
is an equivalence, where \(\mathbf{Vect}_{\mathrm{fd}}\) denotes finite-dimensional vector spaces.
Equivalently, every finitely strongly generated module is a finite direct sum of copies of \(\mathbf V\).
\end{theorem}
\begin{proof}
For \(M\in\mathbf V\on{-mod}^{\mathrm{fin}}\), Lemma~\ref{lem:trivial-top-connection} and Theorem~\ref{th:ACF} give
\[
\widetilde M_X\cong\cdo{\wedge E}\otimes S
\]
for a finite-dimensional vector space \(S\).
Apply the same argument to \(\mf F^*M\) to obtain
\(\widetilde{\mf F^*M}_X\cong\cdo{\wedge E}\otimes S'\).

Taking global sections and using Theorem~\ref{th:no-support}, the natural adjunction morphisms give
\[
\alpha:M\longrightarrow\mathbf V\otimes S,\qquad
\beta:M\longrightarrow\mf F^{-1*}(\mathbf V\otimes S')
\cong\mathbf V\otimes S'.
\]
For a sheaf module \(\cdo{\wedge E}\otimes S\), the counit of localization and global sections is an isomorphism by Theorem~\ref{th:no-support}. The adjunction identities therefore show that \(\alpha\) becomes an isomorphism after localization on \(X\); the Fourier twist of \(\beta\) does so as well.
For \(K:=\ker(\alpha,\beta)\), exactness of localization therefore gives
\(\widetilde K_X=\widetilde{\mf F^*K}_X=0\).
Faithfulness implies \(K=0\), so \(M\) embeds into
\(\mathbf V\otimes(S\oplus S')\).
The vacuum module remains simple after forgetting parity: a proper ungraded submodule and its parity conjugate would give a direct-sum decomposition of the graded-simple module \(\mathbf V\), contradicting its one-dimensional canonical weight-zero space. Thus \(\mathbf V\otimes(S\oplus S')\) is semisimple, and its submodule \(M\) is a finite direct sum of copies of \(\mathbf V\).
\end{proof}

\begin{theorem}\label{th:C1-cofinite}
We have
\[
\operatorname{Ext}^1_{\mathbf V\on{-mod}}
(\mathbf V,\mathbf V)=0.
\]
\end{theorem}
\begin{proof}
Consider an extension
\[
0\longrightarrow\mathbf V\xrightarrow{\iota}M
\xrightarrow{q}\mathbf V\longrightarrow0.
\]
Local finiteness of the commuting operators \(L_0,L'_0\) is preserved under extensions.
The joint generalized eigenspace decomposition therefore allows us to choose a lift \(u\in M\) of the vacuum with joint generalized weight \((0,0)\).
The vectors \(u\) and \(\iota(\vac)\) strongly generate \(M\): for \(m\in M\), put \(a=q(m)\in\mathbf V\); then
\[
m-a_{(-1)}u\in\iota(\mathbf V),
\qquad
\iota(b)=b_{(-1)}\iota(\vac)\quad(b\in\mathbf V).
\]
Thus \(M\in\mathbf V\on{-mod}^{\mathrm{fin}}\).
By Theorem~\ref{th:equivalence}, this category is equivalent to the semisimple category \(\mathbf{Vect}_{\mathrm{fd}}\), so the extension splits.
\end{proof}

\subsection{Supercharacter formula}
Set \(SF(\g):=\bigcap_i\ker b^{\xi_i}_{(0)}\subset\mc E(\g\oplus\g^*)\). This is the rank-\(\dim\g\) symplectic fermion vertex superalgebra generated by \(b^\xi,\partial c^{\xi^*}\), graded by \eqref{eq:cohomological-grading}. The relative complex is
\begin{equation}\label{eq:description-of-relative-complex-via-SF}
C^{\infty/2+i}(\affg_{-\kappa_\g},\g,M)
=(M\otimes SF^i(\g))^\g,
\end{equation}
with \(\g\)-action given by \eqref{eq:level-zero-action}. Let \(P_{\mathrm{mat}}\) denote parity on \(\mathcal A\), extended trivially to the ghost factor. By the BRST vanishing and Euler--Poincar\'e principle,
\[
\on{sch}_{\mathbf V}(q)
=q^{-c_{\mathbf V}/24}\sum_i(-1)^i
\on{tr}_{(\mathcal A\otimes SF^i(\g))^\g}
\bigl((-1)^{P_{\mathrm{mat}}}q^{(L_{\mathrm{can}})_0}\bigr),
\qquad c_{\mathbf V}=-\dim W_0+\dim W_1-2\dim\g.
\]
Here ghost parity equals cohomological degree modulo two, so it is already accounted for by \((-1)^i\). Let \(\Delta\) be the root system, \(r=\on{rk}\g\), and \(\on{Wt}(W_i)\) the weight multisets. Before taking invariants, the Euler--Poincar\'e supercharacter is
\[
\begin{aligned}
\mathcal I(t,q)={}&q^{-c_{\mathbf V}/24}
\prod_{n\geq0}
\frac{\prod_{\nu\in\on{Wt}(W_1)}(1-t^\nu q^{n+1/2})(1-t^{-\nu}q^{n+1/2})}
{\prod_{\mu\in\on{Wt}(W_0)}(1-t^\mu q^{n+1/2})(1-t^{-\mu}q^{n+1/2})}\\
&\times\prod_{n\geq1}\left((1-q^n)^{2r}
\prod_{\alpha\in\Delta}(1-t^\alpha q^n)^2\right).
\end{aligned}
\]
Weyl integration, with normalized Haar measure on the maximal torus \(T_K\) of a compact form \(K\), gives \(\on{sch}_{\mathbf V}(q)=|W_\g|^{-1}\int_{T_K}\mathcal I(t,q)\prod_{\alpha\in\Delta}(1-t^\alpha)\,dt\),
where $W_\g$ is the Weyl group of $\g$.

For \(|q|<1\), put \(\theta(x;q)=\prod_{n\geq0}(1-xq^n)(1-x^{-1}q^{n+1})\) and \(\eta(q)=q^{1/24}\prod_{n\geq1}(1-q^n)\). Combining the  root factors with the Weyl Jacobian gives the following.
\begin{theorem}\label{th:character}
The supercharacter of \(\mathbf V\) is
\begin{equation}\label{eq:int}
\on{sch}_{\mathbf V}(q)
=\frac{q^{(-c_{\mathbf V}-2r)/24}\eta(q)^{2r}}{|W_\g|}
\int_{T_K}
\frac{\prod_{\nu\in\on{Wt}(W_1)}\theta(t^\nu q^{1/2};q)
\prod_{\alpha\in\Delta}\theta(t^\alpha q;q)}
{\prod_{\mu\in\on{Wt}(W_0)}\theta(t^\mu q^{1/2};q)}\,dt.
\end{equation}
\end{theorem}
 For related integrals, see \cite{AKM,BDF,PP}.

\section{Applications to classical groups}\label{sec:classical}
We now take \(\g\) to be a classical simple Lie algebra with normalized invariant form \(\kappa_0\). Write \(V^k(\g):=V^{k\kappa_0}(\g)\) and denote its simple quotient by \(L_k(\g)\). Throughout, \(\on{Mat}(a,b)=\on{Hom}(\C^b,\C^a)\), with the group $G$ acting from the left.

\subsection{\texorpdfstring{The case \(G=SL_n\)}{The special linear case}}\label{section6.1}
Assume \(n\geq2\). Here \(h^\vee=n\) and \(\kappa_{\on{Mat}(n,m)}=\frac{m}{2n}\kappa_\g\). Choose \(M,N\geq0\) with \(M-N=2n\), and set
\[
W_0=\on{Mat}(n,M),\qquad W_1=\on{Mat}(n,N).
\]
Then \eqref{eq:level-condition} holds. Let \(Z\subset W_0\) be the determinantal variety of matrices of rank less than \(n\), and put \(U=W_0\setminus Z\). The codimension 
of $Z$ is \(M-n+1\geq2\), and \(G\) acts freely on \(U\) with closed orbits, giving \ref{A1} and \ref{A2}. The moment map is
\[
\mu(A,B)=AB-\frac{\on{tr}(AB)}nI_n,
\qquad
\mu^{-1}(0)=\{(A,B):AB=\lambda I_n\text{ for some }\lambda\in\C\}.
\]
\begin{lemma}\label{lem:codim}
Assumptions \ref{A3} and \ref{A4} hold.
\end{lemma}
\begin{proof}
By Lemma~\ref{lem:isoHAcond}(2), it suffices to show that the rank-deficient locus in every fiber \(\mu^{-1}(0)\cap(\{A\}\times W_0^*)\) has codimension at least two. After changing bases, write \(A=(I_k\oplus0\ \ 0)\), where \(k=\on{rank}A\).
If \(k=n\), then \(B=\binom{\lambda I_n}{D}\), with \(D\in\on{Mat}(M-n,n)\). Rank deficiency requires \(\lambda=0\) and \(\on{rank}D<n\), so its codimension is \(M-2n+2\).
If \(k<n\), then \(AB=0\) and \(B=\binom0D\), with \(D\in\on{Mat}(M-k,n)\). The rank-deficient locus has codimension \(M-n-k+1\geq M-2n+2\geq2\).
\end{proof}

By \eqref{eq:cc} and Theorems~\ref{Th:vanishing-and-simplicity} and \ref{th:no-support},
\[
\mathbf V=H^{\infty/2+0}(\affg,\g,\mc D^{ch}_{\wedge F}(W_0))
\cong\Gamma(X,\cdo{\wedge E})
\]
is simple, has canonical central charge \(2-4n^2\), and has the supercharacter given in Theorem~\ref{th:character}.

We prove Conjecture~\ref{mainconj}. Write \(\beta^{ij},\gamma^{ij},b^{is},c^{is}\) for the generators of \(\mc A\), with singular OPEs
\[
\beta^{ij}(z)\gamma^{kl}(w)\sim\frac{\delta_{ik}\delta_{jl}}{z-w},
\qquad b^{is}(z)c^{kt}(w)\sim\frac{\delta_{ik}\delta_{st}}{z-w}.
\]
Bars denote their images in \(R_{\mc A}\cong\mc O(T^*W_0\times\Pi T^*W_1)\). For \(V=\C^n\), the first fundamental theorem for \(SL_n\) (see \cite{LS25}) says that \(\mc O(T^*W_0)^G\) is generated by
\begin{equation}\label{eq:generators-of-classical-invairaint-rings-sl_n}
\begin{gathered}
\bar E_{ij}=\sum_{a=1}^n\bar\beta^{ai}\bar\gamma^{aj},\qquad
\bar D(I)=\det(\bar\gamma^{a i_b})_{a,b=1}^n,\qquad
\bar D'(I)=\det(\bar\beta^{a i_b})_{a,b=1}^n,
\end{gathered}
\end{equation}
where \(1\leq i,j\leq M\) and \(I=(i_1,\ldots,i_n)\).
\begin{prop}\label{prop:classicalinv}
We have \(\mc O(T^*X)\cong\mc O(T^*W_0)^G/\widetilde I\), where \(\widetilde I\) is generated by
\begin{align}
&\sum_{k=1}^M\bar E_{ik}\bar E_{kj}-\frac{\on{tr}(\bar E)}n\bar E_{ij},\label{eqn:invI1}\\
&\sum_{k=1}^M\bar E_{kj}\bar D(i_1,\ldots,i_{n-1},k)
-\frac{\on{tr}(\bar E)}n\bar D(i_1,\ldots,i_{n-1},j),\label{eqn:invI2}\\
&\sum_{k=1}^M\bar E_{jk}\bar D'(i_1,\ldots,i_{n-1},k)
-\frac{\on{tr}(\bar E)}n\bar D'(i_1,\ldots,i_{n-1},j),\label{eqn:invI3}
\end{align}
for \(1\leq i,j\leq M\) and \(\{i_1,\ldots,i_{n-1}\}\subset\{1,\ldots,M\}\).
\end{prop}
\begin{proof}
Put \(R=\mc O(T^*W_0)\), \(A=(a_1,\ldots,a_M)\), \(B=(b_1,\ldots,b_M)^t\), where \(a_j=(\bar\gamma^{ij})_i\in R\otimes V\) and \(b_j=(\bar\beta^{ij})_i\in R\otimes V^*\). Then \(\bar E=BA\) and
\[
Z:=AB-\frac{\on{tr}(\bar E)}nI_V
=\sum_a\mu^*(\xi_a)\otimes\xi^a,
\]
where the bases are dual for the trace form. Since \(\mc O(\mu^{-1}(0))^G=R^G/I_0^G\), it suffices to determine \(I_0^G\). Reductivity and projection onto the traceless part give surjections
\[
(R\otimes\on{End}V)^G\twoheadrightarrow(R\otimes\g)^G
\twoheadrightarrow I_0^G,
\qquad F\longmapsto\on{tr}_V(FZ).
\]
Thus covariant generators give ideal generators by pairing with \(Z\).

The map \(F\mapsto\xi(F(a,b)u)\) identifies the covariant module with the part of \(\mc O(V^{\oplus(M+1)}\oplus(V^*)^{\oplus(M+1)})^G\) of bidegree \((1,1)\) in the extra variables \(u,\xi\). Define polynomial covariants by
\[
d_J(v)=\det(a_{j_1},\ldots,a_{j_{n-1}},v),\qquad
\xi(d'_I)=\det(b_{i_1},\ldots,b_{i_{n-1}},\xi).
\]
The first fundamental theorem gives the generators
\(I_V,a_j\otimes b_i,a_j\otimes d_I,d'_I\otimes b_j,d'_I\otimes d_J\).
Pairing the identity with \(Z\) gives 
 zero, and the next three types give
\begin{align*}
b_i(Za_j)&=\sum_k\bar E_{ik}\bar E_{kj}-\frac{\on{tr}(\bar E)}n\bar E_{ij},\\
d_I(Za_j)&=\sum_k\bar E_{kj}d_I(a_k)-\frac{\on{tr}(\bar E)}n d_I(a_j),\\
b_j(Zd'_I)&=\sum_k\bar E_{jk}b_k(d'_I)-\frac{\on{tr}(\bar E)}n b_j(d'_I).
\end{align*}
Finally, for \(C_{I,J}=(\bar E_{i_rj_s})_{r,s=1}^{n-1}\), the cofactor identity
\[
d'_I\otimes d_J=\det(C_{I,J})I_V
-\sum_{r,s=1}^{n-1}\on{Cof}_{rs}(C_{I,J})\,a_{j_s}\otimes b_{i_r}
\]
expresses the last type in terms of the first two. Pairing with \(Z\) therefore gives exactly the ideal generated by \eqref{eqn:invI1}--\eqref{eqn:invI3}.
\end{proof}

Recall the surjection \(\mc C:=\mc A^{\g[t]}\twoheadrightarrow\mathbf V\). By \cite[Theorem 4.4]{LS24}, \(\mc C\) is strongly generated by
\begin{equation}\label{eq:gl(2n+N|N)}
\begin{aligned}
E_{ij}&=\sum_a:\beta^{ai}\gamma^{aj}:,&
\mathring E_{st}&=\sum_a:b^{as}c^{at}:,\\
P_{it}&=\sum_a:\beta^{ai}c^{at}:,&
\mathring P_{sj}&=\sum_a:b^{as}\gamma^{aj}:,
\end{aligned}
\end{equation}
and the supersymmetric determinants
\[
D(I;J)=\det(\gamma^{i_1},\ldots,\gamma^{i_r},c^{j_{r+1}},\ldots,c^{j_n}),\quad
D'(I;J)=\det(\beta^{i_1},\ldots,\beta^{i_r},b^{j_{r+1}},\ldots,b^{j_n}).
\]
Here \(0\leq r\leq n\); the even indices are distinct elements of \(\{1,\ldots,M\}\), while the odd indices lie in \(\{1,\ldots,N\}\) and may repeat. Columns have length \(n\), and determinants are interpreted with the super sign conventions of \cite{LS24}. 
In particular,
\[
\begin{aligned}
D(J;s)=(-1)^{n-1}\sum_a C_a(J)c^{as},\qquad 
D'(J;s)=(-1)^{n-1}\sum_a C'_a(J)b^{as},
\end{aligned}
\]
where \(J=(i_2,\ldots,i_n)\) and the cofactors are fixed by
\(D(j,J)=\sum_a\gamma^{aj}C_a(J)\) and
\(D'(j,J)=\sum_a\beta^{aj}C'_a(J)\).
Equivalently, \(D(J;s)=D(i_2,\ldots,i_n;s)\) and
\(D'(J;s)=D'(i_2,\ldots,i_n;s)\), with the odd column last.
Moving it to the first position multiplies either determinant by
\((-1)^{n-1}\).
We abbreviate \(D(I;\varnothing)\) and \(D'(I;\varnothing)\) to \(D(I)\) and \(D'(I)\).
The currents \eqref{eq:gl(2n+N|N)} give homomorphisms
\(V^{-n}(\mf{gl}_{M|N})\to\mc C\twoheadrightarrow\mathbf V\), conformal with respect to the canonical conformal vectors by \cite{LS24} and Theorem~\ref{surj:generator}.

\begin{proposition}\label{prop:relations-super}
The kernel of \(\mc C\twoheadrightarrow\mathbf V\) contains the following vectors, where \(J=(i_2,\ldots,i_n)\) and all even indices lie in \(\{1,\ldots,M\}\):
\begin{align}
&\sum_{t=1}^M:E_{it}E_{tj}:+\sum_{s=1}^N:P_{is}\mathring P_{sj}:
-\frac1n:HE_{ij}:+n\partial E_{ij}-\delta_{ij}L^{\mathcal C},
\label{re:3}\\
&\sum_{t=1}^M:E_{tj}D(t,J):
+(-1)^n\sum_{s=1}^N:\mathring P_{sj}D(J;s):
-\frac1n:HD(j,J):-\partial D(j,J),\label{re:1}\\
&\sum_{t=1}^M:E_{jt}D'(t,J):
-(-1)^n\sum_{s=1}^N:P_{js}D'(J;s):
-\frac1n:HD'(j,J):+\partial D'(j,J).\label{re:2}
\end{align}
\end{proposition}
\begin{proof}
For \(I=(i_1,\ldots,i_n)\), define \(D_\xi(I)\) by replacing the first column of \(D(I)\) with \(J^\xi_{(0)}\gamma^{i_1}\); define \(D'_\xi(I)\) similarly using \(\beta\). With trace-dual bases \(\{\xi_a\},\{\xi^a\}\), direct calculation gives
\begin{align}
Q_{(0)}\sum_a:b^{\xi_a}D_{\xi^a}(j,J):
&=\sum_t:E_{tj}D(t,J):+(-1)^n\sum_s:\mathring P_{sj}D(J;s):\nonumber\\
&\quad-\frac1n:HD(j,J):-\partial D(j,J),\label{eqn:EL1}\\
Q_{(0)}\sum_a:b^{\xi_a}D'_{\xi^a}(j,J):
&=-\sum_t:E_{jt}D'(t,J):+(-1)^n\sum_s:P_{js}D'(J;s):\nonumber\\
&\quad+\frac1n:HD'(j,J):-\partial D'(j,J),\label{eqn:EL2}\\
Q_{(0)}\sum_{a,k}:b^{\xi_a}:(J^{\xi^a}_{(0)}\beta^{ki})\gamma^{kj}::
&=-\sum_t:E_{it}E_{tj}:-\sum_s:P_{is}\mathring P_{sj}:\nonumber\\
&\quad+\frac1n:HE_{ij}:-n\partial E_{ij}+\delta_{ij}L.\label{eqn:EL3}
\end{align}
The assertion follows from \eqref{eq:re-expression-of-onformal-vector}.
\end{proof}

\begin{proposition}\label{pro:reduce-to-even-case}
There is a surjection \(\mc O(T^*W_0)^G\twoheadrightarrow(R_{\mc C})_{\mathrm{red}}\).
\end{proposition}
\begin{proof}
The odd generators \(P_{it},\mathring P_{sj}\) are nilpotent in \(R_{\mc C}\). So are \(D(I;J),D'(I;J)\) with fewer than \(n\) even columns: sufficiently high normally ordered powers vanish, since they involve odd variables and have no internal contractions.
For \(N\geq2\), the currents \(\mathring E_{st}\) with \(s\neq t\) and \(\mathring E_{ss}-\mathring E_{s+1,s+1}\) generate the lisse algebra \(L_n(\mf{sl}_N)\), so their images are nilpotent. For each \(1\leq s\leq N\), the fields \(:b^{1s}\cdots b^{ns}:\), \(:c^{1s}\cdots c^{ns}:\) and \(\sum_a:b^{as}c^{as}:\) generate a rank-one lattice vertex superalgebra. Its lisse property implies nilpotency of the last current, and hence of \(\sum_s\mathring E_{ss}\).
Thus the reduced ring is generated by the images of \(E_{ij},D(I),D'(I)\). The inclusion of the purely even commutant induces the required surjection, using
\(R_{\cdo{W_0}(W_0)^{\g[t]}}\cong\mc O(T^*W_0)^G\) from \cite[Corollary 4.2]{LS24}.
\end{proof}

\begin{theorem}\label{th:ass-variety-and-nilpotency-index}
The image \(\omega\) of \(L\) in \(R_{\mathbf V}\) is nilpotent, and
\[
(R_{\mathbf V})_{\mathrm{red}}\cong\mc O(T^*W_0\ssslash_0G).
\]
Thus Conjecture~\ref{mainconj} holds and \(\mathbf V\) is quasi-lisse.
\end{theorem}
\begin{proof}
In \((R_{\mathbf V})_{\mathrm{red}}\), write \(\widetilde\omega\) for the image of \(\omega\) and \(\widetilde E=(\widetilde E_{ij})\) for the matrix of current images. The odd terms vanish, so \eqref{eqn:EL3} gives
\begin{equation}\label{eq:square-of-E}
\widetilde E^2=\frac{\on{tr}(\widetilde E)}n\widetilde E+\widetilde\omega I_M.
\end{equation}
Proposition~\ref{pro:reduce-to-even-case} also preserves the classical relation \(\det(BA)=0\), since \(M>n\). Hence
\[
\widetilde\omega^M
=\det\!\left(\widetilde E\left(\widetilde E-\frac{\on{tr}(\widetilde E)}nI_M\right)\right)=0.
\]
Reducedness gives \(\widetilde\omega=0\), proving nilpotency of \(\omega\).
Proposition~\ref{pro:reduce-to-even-case} and Theorem~\ref{thm:surjection-to-classical-invairiant-rings} give surjections
\[
\mc O(T^*W_0)^G\twoheadrightarrow(R_{\mathbf V})_{\mathrm{red}}
\twoheadrightarrow\mc O(T^*W_0\ssslash_0G).
\]
The relations \eqref{re:1}--\eqref{re:3} reduce to \eqref{eqn:invI1}--\eqref{eqn:invI3}. By Proposition~\ref{prop:classicalinv}, the first map therefore factors through the classical reduction. The resulting two surjections in both directions have identity composite, so both are isomorphisms.
\end{proof}

When \(N=0\), we obtain a complete presentation.
\begin{theorem}\label{thm:relation}
Let \(N=0\), so \(M=2n\). The kernel of \(\mc{C}=\cdo{W_0}(W_0)^{\g[t]}\twoheadrightarrow\mathbf V\) is generated by
\begin{align}
&\sum_k:E_{ik}E_{kj}:-\frac1n:HE_{ij}:+n\partial E_{ij}-\delta_{ij}L^{\mathcal C},\label{re:30}\\
&\sum_k:E_{kj}D(k,J):-\frac1n:HD(j,J):-\partial D(j,J),\label{re:10}\\
&\sum_k:E_{jk}D'(k,J):-\frac1n:HD'(j,J):+\partial D'(j,J),\label{re:20}
\end{align}
where \(J=(i_2,\ldots,i_n)\) and all indices range from \(1\) to \(2n\).
\end{theorem}
\begin{proof}
Let \(\mc I\) be the ideal generated by these vectors and put \(V=\mc{C}/\mc I\). It is Fourier-stable, and there is a surjection \(p:V\twoheadrightarrow\mathbf V\). Let \(f=D(1,\ldots,n)\in\mc O(W_0)^G\), so
\(\mf F(f)=\Phi(f)=(-1)^nD'(1,\ldots,n)\). We first show
\begin{equation}\label{eq:omega-is-killed-by-f}
f\widetilde\omega=\mf F(f)\widetilde\omega=0\quad\text{in }R_V,
\end{equation}
where \(\widetilde\omega\) is the image of \(L^{\mathcal C}\). Let \(d_i,d'_i\) be the images of \(D(i,2,\ldots,n),D'(i,2,\ldots,n)\), and write \(d=(d_i)_i,d'=(d'_i)_i\) and \(\widetilde E=(\overline E_{ij})\) in \(R_V\). Relations \eqref{re:10} and \eqref{re:20} give
\[
\widetilde E^t d=\frac{\on{tr}(\widetilde E)}n d,\qquad
\widetilde E d'=\frac{\on{tr}(\widetilde E)}n d'.
\]
Relation \eqref{re:30} gives \eqref{eq:square-of-E} in \(R_V\). Applying it to \(d'\), and its transpose to \(d\), yields \(\widetilde\omega d=\widetilde\omega d'=0\). Since \(d_1=f\) and \(d'_1=(-1)^n\mf F(f)\), this proves \eqref{eq:omega-is-killed-by-f}.

The principal open \(\widetilde X_f\) lies in \(X\), and the local CDO description gives
\begin{equation}\label{eq:induced-hom-for-associated-grade}
\on{gr}_F V_f\twoheadrightarrow\on{gr}_F\mathbf V_f
\cong\mc O(J_\infty(T^*\widetilde X_f)).
\end{equation}
On the diagonal,
this restricts to
\begin{equation}\label{eq:surjection-R_V-to-R_V}
(R_V)_f\twoheadrightarrow(R_{\mathbf V})_f
\cong\mc O(T^*\widetilde X_f)
\cong\mc O(T^*W_0\ssslash_0G)_f.
\end{equation}
By \eqref{eq:omega-is-killed-by-f}, \(\widetilde\omega=0\) after localization. Thus \eqref{re:10}--\eqref{re:30} become the classical defining relations, making \eqref{eq:surjection-R_V-to-R_V} an isomorphism. By \eqref{eq:Li-surjection}, we obtain
\[
\mc O(J_\infty(T^*\widetilde X_f))
\twoheadrightarrow\on{gr}_F V_f
\twoheadrightarrow\on{gr}_F\mathbf V_f
\cong\mc O(J_\infty(T^*\widetilde X_f)).
\]
The composite is the canonical isomorphism, so \(p_f\) is an isomorphism. Fourier transform gives the same conclusion for \(p_{\mf F(f)}\). The Faithfulness Theorem~\ref{th:faithfulness-2} now gives \(V\cong\mathbf V\).
\end{proof}

\subsection{\texorpdfstring{The case \(G=Sp_{2n}\)}{The symplectic case}}\label{section6.2}
Assume \(n\geq1\). Here \(h^\vee=n+1\) and \(\kappa_{\on{Mat}(2n,m)}=\frac{m}{2(n+1)}\kappa_\g\). Choose \(M,N\geq0\) with \(M-N=2n+2\), and set \(W_0=\on{Mat}(2n,M)\), \(W_1=\on{Mat}(2n,N)\). Let \(Z\subset W_0\) be the rank-deficient locus and \(U=W_0\setminus Z\). Then \ref{A1} and \ref{A2} hold, and
\[
\mu^{-1}(0)=\{(A,B):\mathfrak J AB=-(\mathfrak J AB)^t\},\qquad
\mathfrak J=\begin{pmatrix}0&I_n\\-I_n&0\end{pmatrix}.
\]

We verify the hypotheses of
Lemma~\ref{lem:isoHAcond}(2).
Fix \(A\), put \(d=2n\), \(k=\on{rank}A\), and write
\(F_A=\mu^{-1}(0)\cap(\{A\}\times W_0^*)\).
Set \(T=\mathfrak JA\).
The transformations
\((T,B)\mapsto(PTQ,Q^{-1}BP^t)\)
preserve skew-symmetry of \(TB\) and the rank of \(B\).
We may therefore reduce \(T\) to rank normal form, obtaining
\[
F_A\cong
\left\{
\begin{pmatrix}C&0\\D&E\end{pmatrix}
:C^t=-C
\right\},
\]
where \(C\) has size \(k\).
Consider the incidence variety
\(\{(B,[v])\in F_A\times\mathbb P^{d-1}:Bv=0\}\).
Writing \(v=(x,y)\in\C^k\oplus\C^{d-k}\), the condition
\(Bv=0\) imposes \(M-1\) independent linear equations
when \(x\neq0\), and \(M-k\) when \(x=0\).
The corresponding parameter spaces have dimensions
\(d-1\) and \(d-k-1\), respectively.
Hence
\[
\on{codim}_{F_A}(F_A\cap Z^*)\geq M-d
=M-2n\geq2.
\]
Lemma~\ref{lem:isoHAcond}(2) now gives
\ref{A3} and \ref{A4}. The canonical central charge is \(-8n^2-6n\).

The standard module \(V=\C^{2n}\) is self-dual, so \(T^*W_0\cong V^{\oplus2M}\) and \(T^*W_1\cong V^{\oplus2N}\). Define columns of free fields by
\[
X^j=\beta^j,\quad X^{M+j}=\mathfrak J\gamma^j\ (1\leq j\leq M),
\qquad X^{2M+s}=b^s,\quad X^{2M+N+s}=\mathfrak J c^s\ (1\leq s\leq N).
\]
Their OPEs and invariant currents are
\[
X^{si}(z)X^{tj}(w)\sim-\frac{\mathfrak J_{st}C_{ij}}{z-w},\qquad
C=\begin{pmatrix}0&I_M&0&0\\I_M&0&0&0\\0&0&0&I_N\\0&0&-I_N&0\end{pmatrix},
\qquad E^{ij}=\sum_{s,t}\mathfrak J_{st}:X^{si}X^{tj}:.
\]
By \cite{LS24}, these currents strongly generate \(\mc C=\cdo{\wedge F}(W_0)^{\g[t]}\) and give a surjection \(V^{-2n}(\mf{osp}_{2M|2N})\twoheadrightarrow\mc C\). The \(\widehat{\mf{so}}_{2M}\) and \(\widehat{\mf{sp}}_{2N}\) levels are \(-2n\) and \(n\), respectively. Simplicity and Theorem~\ref{surj:generator} give:
\begin{theorem}\label{thm:V=L-sp}
For \(M-N=2n+2\), we have
\[
\mathbf V\cong L_{-2n}(\mf{osp}_{2M|2N})
\cong\Gamma(X,\cdo{\wedge E}).
\]
\end{theorem}
\begin{corollary}\label{thm:Admovic-Milas}
For \(M,N\geq0\) with \(M-N=2n+2\), the homogeneous 
supercharacter satisfies
\[
\operatorname{sch}L_{-2n}(\mathfrak{osp}_{2M|2N})(q)
=
\operatorname{ch}L_{-2n}(\mathfrak{so}_{4n+4})(q).
\]
\end{corollary}
\begin{proof}
Write \(V=\C^{2n}\) for the standard \(G=Sp_{2n}\)-module.
Since \(W_0\cong V^{\oplus M}\) and
\(W_1\cong V^{\oplus N}\), the matter factors in
Theorem~\ref{th:character} cancel in pairs, leaving an expression
depending only on \(M-N\). The central charge also
depends only on this difference.
Comparing with the case \(N=0\) and \(M=2n+2\), for which
\(\mathbf V\cong L_{-2n}(\mathfrak{so}_{4n+4})\),
proves the assertion.
\end{proof}
For \(n=1\), this identity was conjectured by
Adamovi\'c and Milas in \cite[Conjecture 8.1]{AdMi}.

\medskip

Put \(R=\mc O(T^*W_0\times\Pi T^*W_1)\). By \cite{LS24}, \(R_{\mc C}\cong R^G\), generated by \(\bar E^{ij}=\sum_{s,t}\mathfrak J_{st}\bar X^{si}\bar X^{tj}\). Let \(\mu^*_{\mathrm{sup}}:\g\to R\) be the comoment map induced by the affine currents, extending the ordinary comoment map, and let \(I=(\mu^*_{\mathrm{sup}}(\g))\subset R\).
\begin{lemma}\label{lem:reduced-ring-sp}
The ideal \(I^G\subset R^G\) is generated by
\(\sum_{k,l}(C^{-1})_{kl}\bar E^{ik}\bar E^{lj}\), for \(1\leq i,j\leq2M+2N\).
\end{lemma}
\begin{proof}
Apply the covariant argument of Proposition~\ref{prop:classicalinv} and the first fundamental theorem for \(Sp_{2n}\).
\end{proof}
\begin{lemma}\label{lem:reduced-ring-is-sp}
We have \((R^G/I^G)_{\mathrm{red}}\cong\mc O(T^*W_0\ssslash_0G)\).
\end{lemma}
\begin{proof}
The ideal generated by \(\bar E^{ij}\) with \(i>2M\) or \(j>2M\) is nilpotent: its generators involve odd variables. Quotienting by it gives \(\mc O(T^*W_0\ssslash_0G)\), which is reduced by Lemma~\ref{lem:isoHAcond}. This ideal is therefore the nilradical.
\end{proof}

\begin{proposition}\label{prop:lifting-sp}
The kernel of \(\mc C\twoheadrightarrow\mathbf V\) contains
\[
\frac12\sum_{k,l}(C^{-1})_{kl}:E^{ik}E^{lj}:
-\frac{2n+1}{2}\partial E^{ij}-C_{ij}L^{\mathcal C}_{\mathrm{can}},
\qquad 1\leq i,j\leq2M+2N,
\]
where \(L^{\mathcal C}_{\mathrm{can}}=\frac12(L^{\mathcal C}+\mf F(L^{\mathcal C}))\).
\end{proposition}
\begin{proof}
With bases dual for \((\xi,\eta)\mapsto\on{tr}_{\C^{2n}}(\xi\eta)\), the BRST calculation is
\[
\begin{aligned}
Q_{(0)}\sum_{a,s,t}\mathfrak J_{st}:b^{\xi_a}:(J^{\xi^a}_{(0)}X^{si})X^{tj}::
&=\frac12\sum_{k,l}(C^{-1})_{kl}:E^{ik}E^{lj}:\\
&\quad-\frac{2n+1}{2}\partial E^{ij}-C_{ij}L_{\mathrm{can}}.
\end{aligned}
\]
Use \eqref{eq:re-expression-of-onformal-vector}.
\end{proof}

\begin{proposition}\label{pro:nilpotency-omega}
The image \(\omega\) of \(L_{\mathrm{can}}\) in \(R_{\mathbf V}\) is nilpotent.
\end{proposition}
\begin{proof}
Write \(\widetilde E\) for the matrix of current images. Proposition~\ref{prop:lifting-sp} gives \(\widetilde E C^{-1}\widetilde E=2\omega C\). Decompose \(\widetilde E\) into even and odd blocks, with \(\widetilde E_{11}\) of size \(2M\), and write \(C=C_0\oplus C_1\). The ideal \(J\) generated by the finitely many odd entries of \(\widetilde E_{12}\) is nilpotent, and
\[
\widetilde E_{11}C_0^{-1}\widetilde E_{11}\equiv2\omega C_0\pmod J.
\]
Under \(R_{\mc C}\cong R^G\), the even block is the image of a skew Gram matrix of rank at most \(2n<2M\); hence \(\on{Pf}(\widetilde E_{11})=0\). Taking determinants yields
\[
(2\omega)^{2M}\det C_0
\equiv\det(\widetilde E_{11})^2\det(C_0)^{-1}=0\pmod J.
\]
Since \(\det C_0\neq0\), we have \(\omega^{2M}\in J\), proving the assertion.
\end{proof}

\begin{theorem}\label{thm:quasi-lisse-sp}
Conjecture~\ref{mainconj} holds:
\(X_{\mathbf V}\cong T^*W_0\ssslash_0G\).
In particular, \(\mathbf V\cong L_{-2n}(\mf{osp}_{2M|2N})\) is quasi-lisse.
\end{theorem}
\begin{proof}
The natural maps give surjections of reduced Poisson algebras
\[
(R^G)_{\mathrm{red}}\twoheadrightarrow(R_{\mathbf V})_{\mathrm{red}}
\twoheadrightarrow\mc O(T^*W_0\ssslash_0G).
\]
By Proposition~\ref{pro:nilpotency-omega}, the lifted relations reduce to those generating \(I^G\) in Lemma~\ref{lem:reduced-ring-sp}. The first map therefore factors through \((R^G/I^G)_{\mathrm{red}}\). Lemma~\ref{lem:reduced-ring-is-sp} gives surjections in both directions between \((R_{\mathbf V})_{\mathrm{red}}\) and \(\mc O(T^*W_0\ssslash_0G)\), with identity composite. Both are isomorphisms.
\end{proof}
\begin{remark}
We have
\[
T^*W_0\ssslash_0G
\cong\overline{\mathbb O_{[2^{2n},1^{2M-4n}]}}\subset\mf{so}_{2M}.
\]
For \(N=0\), hence \(M=2n+2\), the identification
\(X_{L_{-2n}(\mf{so}_{4n+4})}\cong\overline{\mathbb O_{[2^{2n},1^4]}}\)
was established in \cite{AFK}.
\end{remark}

\subsection{\texorpdfstring{The case \(G=SO_n\) and its \(O_n/SO_n\)-orbifold}{The special orthogonal case and its orbifold}}\label{section6.3}
Assume \(n\geq5\). We have \(h^\vee=n-2\); the trace form satisfies \(\kappa_{\on{Mat}(n,m)}=\frac{m}{n-2}\kappa_\g\). Choose \(M,N\geq0\) with \(M-N=n-2\), and set \(W_0=\on{Mat}(n,M)\), \(W_1=\on{Mat}(n,N)\). We assume \(M\geq n+1\), equivalently \(N\geq3\). Let \(Z\subset W_0\) be the rank-deficient locus and \(U=W_0\setminus Z\). Then \ref{A1} and \ref{A2} hold. 
The canonical central charge is \(-n(2n-3)\).

The zero fiber of the moment map is
\[
\mu^{-1}(0)
=\{(A,B)\in W_0\times W_0^*:AB=(AB)^t\}.
\]
We verify the hypotheses of
Lemma~\ref{lem:isoHAcond}(2).
Fix \(A\) of rank \(k\), and put
\(F_A=\mu^{-1}(0)\cap(\{A\}\times W_0^*)\).
Using the transformations
\((A,B)\mapsto(PAQ,Q^{-1}BP^t)\), we obtain
\[
F_A\cong
\left\{
\begin{pmatrix}C&0\\D&E\end{pmatrix}
:C^t=C
\right\},
\]
where \(C\) has size \(k\), \(D\) has size
\((M-k)\times k\), and \(E\) has size
\((M-k)\times(n-k)\).

Consider
\(\{(B,[v])\in F_A\times\mathbb P^{n-1}:Bv=0\}\).
Writing \(v=(x,y)\in\C^k\oplus\C^{n-k}\), the condition
\(Bv=0\) imposes \(M\) independent linear equations
when \(x\neq0\), and \(M-k\) when \(x=0\).
The corresponding parameter spaces have dimensions
\(n-1\) and \(n-k-1\), respectively.
Thus
\[
\on{codim}_{F_A}(F_A\cap Z^*)\geq M-n+1\geq2.
\]
Lemma~\ref{lem:isoHAcond}(2) therefore gives
\ref{A3} and \ref{A4}.

The standard module \(V=\C^n\) is self-dual. Define
\[
X^j=\beta^j,\quad X^{M+j}=\gamma^j\ (1\leq j\leq M),
\qquad X^{2M+s}=b^s,\quad X^{2M+N+s}=c^s\ (1\leq s\leq N).
\]
Then
\[
X^{si}(z)X^{tj}(w)\sim\frac{\delta_{st}C_{ij}}{z-w},\qquad
C=\begin{pmatrix}0&I_M&0&0\\-I_M&0&0&0\\0&0&0&I_N\\0&0&I_N&0\end{pmatrix},
\qquad E^{ij}=\sum_{s=1}^n:X^{si}X^{sj}:.
\]
Unlike the preceding cases, the classical invariant generators need not all lift to \(\mc C=\cdo{\wedge F}(W_0)^{\g[t]}\). We instead use the \(O_n\)-action: since \(SO_n\) acts trivially on \(\mc C\), it induces an action of \(\Z_2=O_n/SO_n\).

Set \(\bar X=U/O_n\) and \(\bar E=(U\times W_1)/O_n\). The quotient maps \(p:X\to\bar X\) and \(E\to\bar E\) are two-sheeted finite \'etale covers. Descent gives the CDO
\(\cD^{ch}_{\wedge\bar E}:=(p_*\cdo{\wedge E})^{\Z_2}\), and the equivariant global-sections isomorphism yields
\begin{equation}\label{eq:geometric-realization-osp}
\mathbf{V}^{\Z_2}
\cong\Gamma(\bar X,\cD^{ch}_{\wedge\bar E}).
\end{equation}

\begin{theorem}\label{th:Z2-invariant-subalgebra}
For \(n\geq5\), \(N\geq3\), and \(M-N=n-2\), we have
\[
\mathbf V^{\Z_2}\cong L_n(\mf{osp}_{2N|2M}).
\]
The isomorphism identifies \(L_{\mathrm{can}}\) with the Sugawara conformal vector.
\end{theorem}
\begin{proof}
Let \(\mathcal U\subset\mathbf V^{\Z_2}\) be the vertex subalgebra generated by the images of the currents \(E^{ij}\).
Their OPEs give a surjective homomorphism
\[
V^n(\mf{osp}_{2N|2M})\twoheadrightarrow\mathcal U.
\]
The operators \(L_r,L'_r\), \(r\geq-1\), preserve \(\mathcal U\): on each current their positive modes give only currents or scalar multiples of the vacuum, and covariance then gives stability on the generated subalgebra.

Put \(R=\mc O(T^*W_0\times\Pi T^*W_1)\), and let \(I\subset R\) be the super moment-map ideal.
The first fundamental theorem for \(O_n\), with symmetrization in the even variables and antisymmetrization in the odd variables, says that \(R^{O_n}\) is generated by the quadratic pairings \(\bar E^{ij}\).
In particular, \(\mathcal U\) contains both \(\mc O(W_0)^{O_n}\) and \(\mc O(W_0^*)^{O_n}\).
Choose the Gram determinant
\[
f=\det\left(\sum_{s=1}^n\gamma^{sa}\gamma^{sb}\right)_{1\leq a,b\leq n},
\qquad g=\mf F(f).
\]
Then \(f,g\in\mathcal U\), and their nonvanishing loci lie in the free loci for the two polarizations.
Write \(\bar X_f\) for the corresponding principal open in \(W_0\sslash O_n\).
The local CDO description, finite \'etale descent, and cotangent reduction on the free locus give
\[
R_{(\mathbf V^{\Z_2})_f}
\cong\Gamma(\bar X_f,\mc O_{\wedge T^*\bar E})
\cong\bigl((R/I)_f\bigr)^{O_n}.
\]
Exactness of \(O_n\)-invariants identifies the last algebra with
\((R^{O_n}/I^{O_n})_f\), which is generated by the images of the \(\bar E^{ij}\) and \(f^{-1}\).
Consequently
\[
R_{\mathcal U_f}\twoheadrightarrow R_{(\mathbf V^{\Z_2})_f}.
\]
The argument using differential generation and filtered reconstruction in
 Proposition~\ref{prop:generation-criterion} shows that
\(\mathcal U_f\to(\mathbf V^{\Z_2})_f\) is surjective.
The same argument in the dual polarization gives surjectivity after localization at \(g\).

The proof of Theorem~\ref{th:faithfulness-2} applies here with the two \(O_n\)-invariant rings: they lie in \(\mathcal U\), and the required full contraction is supplied by Lemma~\ref{lem:ne0}.
Applying part~(2) to \(\mathcal U\hookrightarrow\mathbf V^{\Z_2}\) therefore gives
\(\mathcal U=\mathbf V^{\Z_2}\).
The fixed-point algebra is simple by the finite-group simplicity theorem of \cite{DM}, whose argument applies equally to vertex superalgebras. Hence it is the simple affine quotient \(L_n(\mf{osp}_{2N|2M})\).

Finally, \(h^\vee=2N-2M-2\), so \(n+h^\vee=2-n\ne0\) and the Sugawara vector exists.
Both it and \(L_{\mathrm{can}}\) make the generating currents primary of weight one. Their difference is therefore central and has canonical weight two; simplicity and the one-dimensional weight-zero space force this difference to vanish.
\end{proof}
Thus \eqref{eq:geometric-realization-osp} gives a geometric realization of the partially integrable simple affine vertex superalgebra \(L_n(\mf{osp}_{2N|2M})\).

\begin{corollary}
Assume \(n\geq5\), \(N\geq3\), and \(M-N=n-2\).
Then
\[
\operatorname{sch}L_n(\mathfrak{osp}_{2N|2M})
=
\operatorname{sch}L_n(\mathfrak{osp}_{6|2n+2}),
\]
where \(\operatorname{sch}\) denotes the homogeneous
supercharacter with respect to the Sugawara conformal vector.
\end{corollary}
\begin{proof}
The identification
\(\mathbf V^{\mathbb Z_2}
\cong L_n(\mathfrak{osp}_{2N|2M})\)
and the vanishing of the higher relative BRST cohomology
allow us to compute the supercharacter by taking
\(O_n\)-invariants in the Euler--Poincar\'e supercharacter
of the BRST complex.
Here taking \(\mathbb Z_2=O_n/SO_n\)-invariants is exact.

For \(g\in O(n)\), the combined contribution of $W_0$  and $W_1$ is 
\[
\prod_{j\geq0}
\det(I_n-q^{j+1/2}g)^{-2(M-N)}.
\]
Indeed, the even and odd matter fields are copies of the
same self-dual \(O_n\)-module \(\mathbb C^n\), so their
contributions cancel in pairs.
The ghost contribution depends only on \(n\), and
the central charge is
\[
c=-n(M-N)-n(n-1)=-n(2n-3).
\]
Consequently, the Euler--Poincar\'e supercharacter,
and its \(O_n\)-invariant part depend only on \(n\)
and \(M-N\).
Comparing \((M,N)\) with \((n+1,3)\) proves the assertion.
\end{proof}

\medskip

Let $I\subset R$ be as in the proof of Theorem \ref{th:Z2-invariant-subalgebra}.
Exactness of invariants gives \((R^G/I^G)^{\Z_2}\cong R^{O_n}/I^{O_n}\).
\begin{lemma}\label{lem:ideal-O_n}
The ideal \(I^{O_n}\subset R^{O_n}\) is generated by
\(\sum_{u,v}(C^{-1})_{uv}\bar E^{iu}\bar E^{vj}\), for \(1\leq i,j\leq2M+2N\).
\end{lemma}
\begin{proof}
Apply the covariant argument of Proposition~\ref{prop:classicalinv} and the first fundamental theorem for \(O_n\).
\end{proof}
\begin{remark}
We have
\[
\on{Spec}(R^{O_n}/I^{O_n})_{\mathrm{red}}
\cong\{A\in\mf{sp}_{2M}:A^2=0,\ \on{rank}A\leq n\}
\cong\overline{\mathbb O_{[2^n,1^{2M-2n}]}}.
\]
The finite quotient map from \(T^*W_0\ssslash_0SO_n\) to this orbit closure is generically two-to-one.
\end{remark}

\begin{proposition}\label{prop:relations-o_n}
The kernel of \(\mc C^{\Z_2}\twoheadrightarrow\mathbf V^{\Z_2}\) contains
\[
-\frac12\sum_{u,v}(C^{-1})_{uv}:E^{iu}E^{vj}:
-\frac{n-1}{2}\partial E^{ij}+C_{ij}L^{\mathcal C}_{\mathrm{can}},
\qquad1\leq i,j\leq2M+2N.
\]
\end{proposition}
\begin{proof}
With bases dual for \((\xi,\eta)\mapsto\on{tr}_{\C^n}(\xi\eta)\), the corresponding BRST identity is
\[
\begin{aligned}
Q_{(0)}\sum_{a,s}:b^{\xi_a}:(J^{\xi^a}_{(0)}X^{si})X^{sj}::
&=-\frac12\sum_{u,v}(C^{-1})_{uv}:E^{iu}E^{vj}:\\
&\quad-\frac{n-1}{2}\partial E^{ij}+C_{ij}L_{\mathrm{can}}.
\end{aligned}
\]
Use \eqref{eq:re-expression-of-onformal-vector}.
\end{proof}
\begin{proposition}\label{pro;nilpotency-O_n}
The image of \(L_{\mathrm{can}}\) in \(R_{\mathbf V^{\Z_2}}\) is nilpotent.
\end{proposition}
\begin{proof}
Write \(e^{ij}\) and \(\omega\) for the images of \(E^{ij}\) and \(L_{\mathrm{can}}\) in $R_{\mathbf V^{\Z_2}}$, respectively.
The currents \(E^{ij}\) with \(i,j>2M\) come from the free fermion system and generate an image of the integrable affine vertex algebra \(L_n(\mf{so}_{2N})\).
This algebra is \(C_2\)-cofinite, so its positive-weight elements have nilpotent images in its \(C_2\)-algebra, and hence in \(R_{\mathbf V^{\Z_2}}\).
In particular, all \(e^{ij}\) with \(i,j>2M\) are nilpotent.
The mixed entries, with exactly one index greater than \(2M\), are odd and therefore have square zero.
Thus the ideal
\[
J=(e^{ij}\mid i>2M\ \text{or}\ j>2M)\subset R_{\mathbf V^{\Z_2}}
\]
is nilpotent.

Modulo \(C_2\), Proposition~\ref{prop:relations-o_n} gives
\[
\sum_{u,v}(C^{-1})_{uv}e^{iu}e^{vj}=2C_{ij}\omega.
\]
Take \(i=2M+1\) and \(j=2M+N+1\), for which \(C_{ij}=1\).
Both factors in every term on the left belong to \(J\), so \(\omega\in J^2\).
Siince $J$ is nilpotnet, so is $\omega$.
\end{proof}
\begin{theorem}\label{th:assp-variety-orbifold}
We have
\[
X_{\mathbf V^{\Z_2}}\cong\mu^{-1}(0)\sslash O_n
\cong(T^*W_0\ssslash_0G)/\Z_2.
\]
In particular, \(L_n(\mf{osp}_{2N|2M})\) is quasi-lisse for \(M-N=n-2\) and \(N\geq3\).
\end{theorem}
\begin{proof}
Repeat the proof of Theorem~\ref{thm:quasi-lisse-sp}, using Lemma~\ref{lem:ideal-O_n} and Propositions~\ref{prop:relations-o_n} and \ref{pro;nilpotency-O_n}.
\end{proof}

\begin{theorem}\label{th:ass-variety-so}
Conjecture~\ref{mainconj} holds: \(X_{\mathbf V}\cong T^*W_0\ssslash_0G\).
\end{theorem}
\begin{proof}
Exactness of \(\Z_2\)-invariants applied to \(\mathbf V\twoheadrightarrow R_{\mathbf V}\) shows that the inclusion \(\mathbf V^{\Z_2}\hookrightarrow\mathbf V\) induces a surjection \(R_{\mathbf V^{\Z_2}}\twoheadrightarrow(R_{\mathbf V})^{\Z_2}\). Reduction and Theorem~\ref{thm:surjection-to-classical-invairiant-rings} give
\[
(R_{\mathbf V^{\Z_2}})_{\mathrm{red}}
\twoheadrightarrow((R_{\mathbf V})_{\mathrm{red}})^{\Z_2}
\twoheadrightarrow\mc O(T^*W_0\ssslash_0G)^{\Z_2}.
\]
The composite is an isomorphism by Theorem~\ref{th:assp-variety-orbifold}. Thus the equivariant surjection \(p:(R_{\mathbf V})_{\mathrm{red}}\twoheadrightarrow\mc O(T^*W_0\ssslash_0G)\) is injective on invariants. For \(a\in\ker p\), writing \(\sigma\) for the involution, we obtain \(a+\sigma(a)=0\), and then \(a^2=0\), since both expressions are invariant elements of the kernel. Reducedness gives \(a=0\), proving the assertion.
\end{proof}

\end{document}